\documentclass[11pt,a4paper]{article}
\usepackage[latin1]{inputenc}
\usepackage{amsmath}
\usepackage{fullpage}
\usepackage{graphicx}
\usepackage{verbatim}
\usepackage{amsthm}
\usepackage{mathrsfs}

  \newtheoremstyle{example}
     {3pt}
     {10pt}
     {}
     {}
     {\itshape}
     {:}
     {.5em}
     {}

\newtheorem{theorem}{Theorem}[section]
\newtheorem{lemma}[theorem]{Lemma}
\newtheorem{proposition}[theorem]{Proposition}
\newtheorem{corollary}[theorem]{Corollary}
\theoremstyle{remark}
\newtheorem{remark}[theorem]{Remark}
\theoremstyle{definition}
\newtheorem{definition}[theorem]{Definition}
\theoremstyle{example}
\newtheorem{example}[theorem]{Example}

\newcommand{\norm}[1]{\left\Vert#1\right\Vert}
\newcommand{\abs}[1]{\left\vert#1\right\vert}
\newcommand{\set}[1]{\left\{#1\right\}}
\newcommand{\Real}{\mathbb R}
\newcommand{\Nat}{\mathbb N}
\newcommand{\Cp}{\mathbb C}
\newcommand{\eps}{\varepsilon}

\newcommand{\brac}[1]{\left(#1\right)}
\newcommand{\bX}{\beta X}
\newcommand{\A}{(A,X,\mu_A)}
\newcommand{\XA}{(A,X_A,\mu_A^{X_A})}
\newcommand{\Abxa}{(A,\beta X_A,\nu_A)}
\newcommand{\BY}{(B,Y,\mu_B)}

\newcommand{\AY}{\brac{A,Y,\mu_A^Y}}
\newcommand{\bXA}{\beta X_A}
\newcommand{\AK}{(A,K,\mu_A^K)}
\newcommand{\AGl}{(A , \Gl (A) , \mu_A )}
\newcommand{\AbGl}{(A, \beta \Gl (A), \mu_A)}
\newcommand{\B}{(B,K,\mu_B)}
\newcommand{\Ab}{A^{\beta}}
\newcommand{\Abb}{A^{\beta \beta}}
\newcommand{\Abx}{(\Ab , \beta X_A , \mu_{\Ab} )}
\newcommand{\AbK}{(\Ab, K , \mu_{\Ab}^K )}

\newcommand{\nhat}{\hat{\mathbb{N}}}

\newcommand{\MA}{(M(A),\bX,\mu_{M(A)})}

\newcommand{\Pm}{\mathrm{Prim}}

\newcommand{\Gl}{\mathrm{Glimm}}

\newcommand{\Cst}{C$^{\ast}$}

\newcommand{\maxt}{\otimes_{\max}}
\newcommand\restr[2]{{
  \left.\kern-\nulldelimiterspace 
  #1 
  \right|_{#2} 
  }}

\newcommand{\gb}{\Gamma^b}
\newcommand{\go}{\Gamma_0}
\usepackage[all]{xy}
\usepackage{amsfonts}
\usepackage{amssymb}
\author{David McConnell\\
\emph{School of Mathematics and Statistics, University of Glasgow }\\
\texttt{david.mcconnell@glasgow.ac.uk}}
\title{A Stone-\v{C}ech Theorem for $C_0(X)$-algebras}
\date{}
\begin{document}
\maketitle
\abstract{For a $C_0(X)$-algebra $A$, we study $C(K)$-algebras $B$ that we regard as compactifications of $A$, generalising the notion of (the algebra of continuous functions on) a compactification of a completely regular space.  We show that $A$ admits a Stone-\v{C}ech-type compactification $\Ab$, a $C(\bX)$-algebra with the property that every bounded continuous section of the C$^{\ast}$-bundle associated with $A$ has a unique extension to a continuous section of the bundle associated with $\Ab$. Moreover, $\Ab$ satisfies a maximality property amongst compactifications of $A$ (with respect to appropriately chosen morphisms) analogous to that of $\bX$.  We investigate the structure of the space of points of $\bX$ for which the fibre algebras of $\Ab$ are non-zero, and partially characterise those $C_0(X)$-algebras $A$ for which this space is precisely $\bX$.
\section*{Introduction}

Bundles, or (semi-) continuous fields arise as a natural way to study the structure of non-simple \Cst -algebras. Indeed, the Gelfand-Naimark representation of a commutative \Cst -algebra may be viewed in this way, and thus suggests the approach of representing a general (i.e. non-commutative) \Cst -algebra as the algebra of continuous sections of a bundle of \Cst -algebras, over a suitably constructed base space.  To this end, Kasparov~\cite{kasparov} introduced the notion of a $C_0 (X)$-algebra, which represents a C$^{\ast}$-algebra $A$ as the section algebra of a bundle over a locally compact Hausdorff space $X$.  This generalised earlier work of Fell~\cite{fell}, Dixmier and Douady~\cite{dixmier_douady}, Dauns and Hofmann~\cite{dauns_hofmann}, Lee~\cite{lee}, and others. Working in the framework of $C_0(X)$-algebras, it is often possible to adapt and generalise the techniques and results from the commutative setting to study more general \Cst -algebras.  Recently $C_0 (X)$-algebras have proved to be extremely useful in advancing the study of the structure and classification of non-simple C$^{\ast}$-algebras, for example in~\cite{blanch_subtrivial},~\cite{kirch_phil},~\cite{blanch_kirch},~\cite{dad_win},~\cite{dad_fd} and~\cite{dad_pen}.
 
 Typically when working with bundles arising from non-unital \Cst -algebras, the natural choice of base space $X$ is non-compact and the original algebra is identified with the algebra of continuous sections of the bundle that vanish at infinity over $X$.  Recent results of Archbold and Somerset~\cite{arch_som_ideals},~\cite{arch_som_inner} motivate the study of the larger algebra of all bounded continuous sections of the bundle, with a view to understanding the structure of multiplier and corona algebras of non-simple \Cst - algebras.  Indeed, examples of this algebra have also arisen in the study of extensions~\cite{ppv1},~\cite{ppv2} and tensor products~\cite{williams_tensor}.  The corresponding object in the commutative case, namely the algebra $C^b (X)$, is often studied by embedding $X$ in a larger, compact space.

A compactification of a space $X$ is a compact Hausdorff space $K$ together with a homeomorphic embedding of $X$ as a dense subspace of $K$.  If $X$ is in addition locally compact, then Gelfand duality gives an equivalent formulation of this notion in the language of commutative \Cst -algebras: a compactification of $X$ is equivalent to a unitisation of the \Cst -algebra $C_0(X)$.  Of particular interest are the minimal (one-point) and maximal (Stone-\v{C}ech) compactifications of $X$, corresponding to the minimal unitisation and the multiplier algebra of $C_0(X)$ respectively.  Our goal here is to study compactifications in the framework of \Cst -bundles.  More precisely, given a bundle over a non-compact space $X$, when and how can it be extended to a bundle over a compactification of $X$ (leaving the fibres over points of $X$ unchanged)?  A similar question for locally trivial bundles with finite-dimensional fibres (i.e. those arising from homogeneous \Cst -algebras) was studied by Phillips in~\cite{phillips_recursive}. Here we consider this problem in the  more general setting of bundles arising from $C_0(X)$-algebras.

Of particular interest is the question of whether or not such a bundle over $X$ extends to a bundle over its Stone-\v{C}ech compactification $\bX$ in such a way that the natural \Cst -bundle analogue of the Stone-\v{C}ech extension property holds: every bounded continuous section over $X$ has a unique extension to a continuous section over $\bX$.  One motivation for studying this question is the desire to obtain a more detailed decomposition of the \Cst -algebra of bounded continuous sections of our original bundle, in line with the classical identification of $C^b(X)$ with $C( \bX)$.

For a locally compact Hausdorff space $X$, a $C_0(X)$-algebra is a \Cst -algebra $A$ which carries the structure of a non-degenerate Banach $C_0(X)$-module.  The maximal ideals of $C_0(X)$ give rise to quotient \Cst -algebras $\{ A_x : x \in X \}$ of $A$, which we regard as the fibres of a bundle of \Cst -algebras over $X$.  Each $a \in A$ then gives rise to a cross-section $X \to \coprod A_x, x \mapsto a(x)$, and there is a natural topology on $\coprod A_x$ such that $A$ is isomorphic to the \Cst -algebra of all continuous sections of this bundle that vanish at infinity on $X$.  The norm functions $x \mapsto \norm{a(x)}$ ($a \in A$) are in general only upper-semicontinuous on $X$; when they are in addition continuous for all $a \in A$ we speak of a \emph{continuous $C_0(X)$-algebra}.  Continuous $C_0(X)$-algebras are equivalent to the continuous fields of \Cst -algebras studied by Fell~\cite{fell}, Dixmier and Douady~\cite{dixmier_douady}, and many others.

When $A$ is a $C_0(X)$-algebra with all fibres nonzero, we define a \emph{compactification} of $A$ to be a $C(K)$-algebra $B$ where $K$ is a compactification of $X$ in the usual sense, $B$ contains $A$ as an essential ideal, and the fibre algebras $B_x$ of $B$ are naturally isomorphic to those of $A$ over points of $X$.  This is equivalent to requiring that the \Cst -bundle over $K$ defined by $B$, when restricted to the dense subspace $X$, coincides with the bundle over $X$ defined by $A$.  

In seeking compactifications of a $C_0(X)$-algebra $A$, it would seem natural at first to take $B=M(A)$ (the multiplier algebra of $A$) and $K= \bX$. While $M(A)$ is indeed a $C( \bX)$-algebra, it often fails to be a compactification of $A$, since the fibre algebras $M(A)_x$ can be much larger than those of $A$.  For example, this occurs whenever the $A_x$ are non-unital, since quotients of $M(A)$ are necessarily unital. We remark that it was shown by Archbold and Somerset that the identity $M(A)_x = M(A_x)$ can also fail in general~\cite{arch_som_ss}, and that when $A$ is a continuous $C_0(X)$-algebra it is often the case that $M(A)$ fails to be continuous~\cite{arch_som_mult}.

We study the \Cst -algebra $\Ab$ of bounded continuous sections associated with a $C_0(X)$-algebra $A$, which in general, lies between $A$ and its multiplier algebra $M(A)$.  In Section~\ref{s:cfn} we show that for any compactification $K$ of $X$, $\Ab$ carries the structure of a $C(K)$-algebra and moreover defines a compactification of $A$ (Theorem~\ref{t:ab}).  We show in the same Theorem that if $A$ is a continuous $C_0(X)$-algebra and $K = \bX$, then $\Ab$ is a continuous $C( \bX )$-algebra.  Moreover, Corollary~\ref{c:cts} shows that $\bX$ is essentially the unique compactification of $X$ with this property.   In general, the operation of constructing $\Ab$ from $A$ is functorial with respect to the natural morphisms between $C_0(X)$-algebras $A$ and $C_0(Y)$-algebras $B$ (Corollary~\ref{c:functorial}), and is a closure operation (Corollary~\ref{c:closure}).  These properties generalise the Stone extension of a continuous map $X \to Y$ to a continuous map $\bX \to \beta Y$, and the fact that $\beta ( \bX) = \bX$ respectively.

In the special case of a trivial $C_0 (X)$-algebra $A = C_0 (X, B)$ for some \Cst -algebra $B$, we have $\Ab = C^b (X,B)$ (Corollary~\ref{c:triv}).  In particular this shows that if $A = C_0(X)$ then $\Ab = C^b( X)$.  More generally, we show  that the $C( \bX)$-algebra $\Ab$, has the property that every bounded continuous section on $X$ of the bundle defined by $A$ extends uniquely to a continuous section on $\bX$ of $\Ab$ (Theorem~\ref{t:ext}). This shows that the bundle arising from the $C(\bX)$-algebra $\Ab$ has an extension property analogous to that of the classical Stone-\v{C}ech compactification. 

The Corona algebra of a $C_0(X)$-algebra is well-known to exhibit interesting and pathological behaviour~\cite{kucerovsky_ng},~\cite{arch_som_ideals},~\cite{arch_som_inner}. With this in mind, we investigate the nature of the fibres of the $C( \bX)$-algebra $\Ab$ at points of the `corona' set $\bX \backslash X$, and in particular, the question of whether or not these fibres are all non-zero \Cst -algebras.  We show in Theorem~\ref{t:nonzero} that $\Ab$ has no zero-algebra fibres whenever $A$ belongs to a large class of $C_0(X)$-algebras, including all $\sigma$-unital, continuous $C_0(X)$-algebras. Interestingly, Example~\ref{e:remote1} shows that this can fail when $A$ is not continuous. Nonetheless, we show in Theorem~\ref{t:notremote} that the set of nonzero fibres of $\Ab$ in $\bX \backslash X$ is at least dense in $\bX \backslash X$ for all $\sigma$-unital $A$. The existence  of zero-algebra fibres of $\Ab$ is closely related to the existence of so-called \emph{remote points} of Stone-\v{C}ech remainders, a topic of significant interest in the field of general topology since the work of Fine and Gilmann in~\cite{fine_gilmann}.

It is well-known that $\bX$ is maximal amongst compactifications of $X$: given any compactification $K$ of $X$, there is a unital, injective $\ast$-homomorphism $C(K) \to C( \bX)$ (commuting with the restriction homomorphisms $f \mapsto \restr{f}{X}$). We show that the compactification $\Ab$ of the $C_0(X)$-algebra $A$ has an analogous maximality and uniqueness property (Theorems~\ref{t:u1} and~\ref{t:unique}).  

The structure of the paper is as follows: in Section~\ref{s:c0x} we introduce the basic definitions and properties of $C_0(X)$-algebras and illustrate how a $C_0(X)$-algebra $A$ defines a \Cst -bundle over $X$.  Since a general $C_0(X)$-algebra $A$ may have some fibres equal to the zero \Cst -algebra, it is often necessary to restrict our attention to the subspace $X_A$ corresponding to the nonzero fibres.  Often, $X$ may be essentially replaced with $X_A$ as a base space, though in general, care is needed since $X_A$ may fail to be locally compact (it is always completely regular and so we may safely speak of $\bXA$).  The notion of a morphism between a $C_0(X)$-algebra $A$ and a $C_0(Y)$-algebra $B$ is made clear in Section~\ref{s:cfn}, where we introduce precise definition of a compactification of a $C_0(X)$-algebra $A$ (allowing for $X_A \subsetneq X$) and its equivalent formulation in the language of \Cst -bundles.

The main results concerning the Stone-\v{C}ech compactification of a $C_0(X)$-algebra appear in Sections~\ref{s:cfn} and~\ref{s:ext}.  In order to account for the possibility of having some fibre algebras equal to zero, the statements of these results are first presented in the framework of more general compactifications $K$ of the space $X_A$, before restricting to the case $K = \bX$.  Section~\ref{s:fibres} concerns the study of the fibre algebras of $\Ab$ of points of the `corona' sets $K \backslash X_A$.  Finally, the maximality and uniqueness results appear in Section~\ref{s:univ}.

\section{Preliminaries on $C_0 (X)$-algebras}
\label{s:c0x}
\begin{definition}
\label{d:c0xalgebra}
Let $X$ be a locally compact Hausdorff space.  A \emph{$C_0(X)$-algebra} is a C$^{\ast}$-algebra $A$ together with a $\ast$-homomorphism $\mu_A : C_0 (X) \rightarrow ZM(A)$ with the property that $\mu_A (C_0 (X) )A = A$. 
\end{definition}

When the space $X$ is compact, the non-degeneracy condition $\mu_A (C_0 (X) )A = A$ above is equivalent to the $\ast$-homomorphism $\mu_A$ being unital.  In this case we will say that $\A$ is a $C(X)$-algebra.

It follows from the  Dauns-Hofmann Theorem~\cite{dauns_hofmann} that there is a $\ast$-isomorphism $\theta_A : C^b ( \Pm (A) ) \rightarrow ZM(A)$ with the property that
\begin{equation}
\label{e:dh}
\theta_A (f) a + P = f(P) (a + P ), \mbox{ for } a \in A , f \in C^b ( \Pm (A)), P \in \Pm (A).
\end{equation}
This gives an equivalent formulation of Definition~\ref{d:c0xalgebra}: a $C_0 (X)$-algebra is a C$^{\ast}$-algebra $A$ together with a continuous map $\phi_A : \Pm (A) \rightarrow X$.  The maps $\mu_A$ and $\phi_A$ are related via $ \mu_A (f) = \theta_A ( f \circ \phi_A )$ for all $f \in C_0 (X)$~\cite[Proposition C.5]{williams}.  We call $\phi_A$ the \emph{base map} and $\mu_A$ the \emph{structure map}.

For clarity we will denote any $C_0 (X)$-algebra $A$ by the triple $(A , X , \mu_A )$.  For  $x \in X$ we define the ideal $J_x$ via
\begin{equation}
\label{e:ix}
J_x = \mu_A \left( \{ f \in C_0 (X) : f(x) = 0 \} \right) A = \bigcap \{ P \in \Pm (A) : \phi_A (P) = x \},
\end{equation}
see~\cite[Section 2]{nilsen_bundles} for example.

We do not require that the base map $\phi_A : \Pm (A) \rightarrow X$ be surjective, or even that $\phi_A ( \Pm (A) )$ be dense in $X$.  It is shown in~\cite[Corollary 1.3]{arch_som_mult} that $\phi_A$ has dense range if and only if the structure map $\mu_A$ is injective.

If $x \in X \backslash \mathrm{Im} ( \phi_A )$, then we may still define the ideal $J_x $ of $A$ via \\$J_x=\mu_A \left( \{ f \in C_0 (X) : f(x) = 0 \} \right) A$; it is shown in~\cite[\S 1]{arch_som_mult} that $J_x = A$ for all such $x$.  This is consistent with our second definition of $J_x$ in~(\ref{e:ix}), when we regard the intersection of the empty set $ \{ P \in \Pm (A) : \phi_A (P) = x \}$ of ideals of $A$ as $A$ itself.

\begin{definition}
Let $\A$ be a $C_0(X)$-algebra. Define the subset $X_A$ of $X$ to be
\[
X_A = \mathrm{Im} ( \phi_A ) = \{ x \in X : J_x \neq A \}.
\]
\end{definition}

For each $x \in X$ let $A_x = A/ J_x$ be the quotient \Cst -algebra, and for $a \in A$ let $a(x) = a + J_x \in A_x$ be the image of $a$ under the quotient $\ast$-homomorphism.  Then the following properties of $\A$ are well-known, see~\cite[Proposition C.10]{williams} or~\cite{nilsen_bundles} for example:
\begin{enumerate}
\item[(i)] $\norm{a} = \sup_{x \in X} \norm{a(x)}$ for all $a \in A$, 
\item[(ii)] for all $a \in A, f \in C_0(X)$ and $x \in X$ we have
\[
(\mu_A(f)a)(x)=f(x)a(x)
\]
\item[(iii)] the function $x \mapsto \norm{a(x)}$ is upper-semicontinuous and vanishes at infinity on $X$.
\end{enumerate}
As a consequence, we may regard $A$ as a \Cst -algebra of cross-sections $X \rightarrow \coprod_{x \in X} A_x$, identifying $a \in A$ with $x \mapsto a(x)$.  Note that under this identification, property (ii) shows that the $\ast$-homomorphism $\mu_A : C_0 (X) \rightarrow ZM(A)$ is given by pointwise multiplication of sections by functions in $C_0(X)$.  Note that for $x \in X \backslash X_A$, the fibre algebras $A_x$ are zero, and we have
\begin{equation}
\label{e:xa}
X_A = \set{ x \in X : A_x \neq 0}
\end{equation}

 In fact, there is a unique topology on $\coprod_{x \in X} A_x$ which defines an upper-semicontinuous \Cst -bundle in such a way that $A$ is canonically isomorphic to the \Cst -algebra of \emph{all} continuous cross-sections of this bundle that vanish at infinity on $X$~\cite[Theorem C.26]{williams}.  While we will not explicitly define the \Cst -bundle topology on $\coprod_{x \in X} A_x$,  it will be useful to identify those cross-sections $X \rightarrow \coprod_{x \in} A_x$ that are continuous with respect to this topology.  Definition~\ref{d:cts} of a continuous cross-section $X \rightarrow \coprod_{x \in X} A_x$ is easily seen to be equivalent to the one used in~\cite[Theorem C.25]{williams} and~\cite[Definition 5.1]{ara_mat_sheaves}, and generalises the definitions used by Fell~\cite{fell} and Dixmier~\cite[Chapter 10]{dix} in the context of continuous fields.  

\begin{definition}
\label{d:cts}
Let $\A$ be a $C_0(X)$-algebra and let $Y \subseteq X$.  We say that a cross-section $b : Y \rightarrow \coprod_{x \in Y} A_x$ is \emph{continuous} (with respect to $\A$) if for all $x \in Y$ and $\eps >0$ there is a neighbourhood $U$ of $x$ in $Y$ and an element $a \in A$ such that
\[
\norm{b(y)-a(y)} < \eps \mbox{ for all } y \in U.
\]
Further, we define $\Gamma (\A)$ to be the collection of all sections $b: X_A \to \coprod_{x \in X_A} A_x$ that are continuous with respect to $\A$. Then $\Gamma (\A)$ is a $\ast$-algebra with respect to pointwise operations.  The $\ast$-subalgebra $\gb (\A)$ (respectively $\go (\A)$) consisting of those continuous sections $b$ for which the norm-function $x \mapsto \norm{ b(x) }$ is bounded (respectively vanishes at infinity on $X_A$), equipped with the supremum norm is a \Cst -algebra.
\end{definition}
When there is no ambiguity regarding the $C_0(X)$-algebra structure on a given \Cst -algebra $A$, we shall write $\Gamma (A), \gb (A)$ and $\go (A)$ respectively to denote the section algebras of Definition~\ref{d:cts}.  It is easily seen that for $b \in \Gamma (A)$, the norm function $x \mapsto \norm{b(x)}$ is upper-semicontinuous on $X_A$. Moreover, when $X_A$ is compact then it is clear that $\Gamma_0 (A) = \gb (A) = \Gamma (A)$.  

\begin{remark}
Usually when considering a $C_0(X)$-algebra $\A$ as an algebra of cross-sections, elements of $\go (A)$ are regarded as continuous cross-sections $X \to \coprod_{x \in X} A_x$ (rather than being defined on $X_A$ as in Definition~\ref{d:cts}).  Note that since $A_x = \set{0}$ for $x \in X \backslash X_A$ and $x \mapsto \norm{a(x)}$ is upper-semicontinuous on $X$, it is clear that a cross-section $a:X \to \coprod_{x \in X} A_x$ is continuous with respect to $\A$ on $X$ if and only if it is continuous on $X_A$. Hence the definition of $\go (A)$ in Definition~\ref{d:cts} is equivalent to the one used in~\cite[Appendix C]{williams}.

The definition of $\gb (A)$, however, does depend on whether sections are defined on $X$ or on $X_A$ (see Example~\ref{e:notsurjective}). Hence using $X_A$ rather than $X$ allows us flexibility in changing the ambient space of $X_A$ without affecting $\gb (A)$.
\end{remark}

The following Theorem recalls some known results about $\go (A):$
\begin{theorem}
\label{t:go}
Let $\A$ be a $C_0(X)$-algebra.  Then the section algebra $\go (A)$ has the following properties:
\begin{enumerate}
\item[(i)] The natural action of functions in $C_0(X)$ on sections in $\go (A)$ by pointwise multiplication, equips $\go (A)$ with the structure of a $C_0(X)$-algebra,  with fibres given by $\go (A)_x \cong A_x$~\cite[Proposition C.23]{williams}.
\item[(ii)] Identifying each $a \in A$ with the cross section $x \mapsto a(x)$ defines a $C_0(X)$-linear $\ast$-isomorphism $A \rightarrow \go (A)$~(\cite{nilsen_bundles},~\cite[Theorem C.26]{williams}).
\item[(iii)] Let $B$ be a \Cst -subalgebra of $\go (A)$ such that
\begin{enumerate}
\item[(a)]  for all $f \in C_0(X)$ and $b \in B$ the section $f \cdot b$, where $(f \cdot b)(x) = f(x) b(x)$ for all $x \in X$, belongs to $B$, and
\item[(b)] for each $x \in X$ and $c \in A_x$ there is some $b \in B$ with $b(x)=c$,
\end{enumerate}
then $B=\go (A)$~\cite[Proposition C.24]{williams}.
\end{enumerate}
\end{theorem}

Let $\A$ be a $C_0(X)$-algebra with base map $\phi_A : \Pm (A) \to X$, and let $Y$ be any locally compact Hausdorff space containing (a homeomorphic image of) $X_A$.  Then we may regard $\phi_A$ as a map from $\Pm (A)$ to $Y$. Hence $A$ defines a $C_0(Y)$-algebra with this base map.

Moreover, as a subspace of the locally compact Hausdorff space $X$, $X_A$ is completely regular.  In particular, $X_A$ admits a homeomorphic embedding into its Stone-\v{C}ech compactification $\beta X_A$.  Thus the above remarks apply with $Y= \beta X_A$.  It shall be convenient to fix some notation to describe this process.

\begin{definition}
\label{d:c0y}
Let $\A$ be a $C_0(X)$-algebra and let $Y$ be a locally compact Hausdorff space containing a homeomorphic image of $X_A$.  The $C_0(Y)$-algebra constructed by regarding $\phi_A$ as a map $\Pm (A) \to Y$ shall be denoted by $ \AY $. In the particular case where $Y= \beta X_A$, we shall set
\[
\Abxa : = \brac{ A , \beta X_A , \mu_A^{\bXA} }.
\]  
\end{definition}
Note that for $x \in X_A$, the ideals $J_x$ of~\eqref{e:ix} are independent of the ambient space of $X_A$. In particular, for each $a \in A$ the cross-section $a : X_A \to \coprod_{x \in X_A} A_x$, $x \mapsto a(x)$ is unchanged by the process of varying the base space.  In particular, the definitions of $\go (A)$ and $\gb (A)$ are unambiguous.

When $X_A$ is locally compact, we may take the base space $Y$ of Definition~\ref{d:c0y} to be $X_A$. Hence $\XA$ is a $C_0(X_A)$-algebra, with (nonzero) fibres $A_x$ for all $x \in X_A$.  Note that in this case $X_A$ is open in $\bXA$, and so $C_0(X_A)$ is the ideal $\set{ f \in C( \bXA) : \restr{f}{\bXA \backslash X_A} \equiv 0 }$.  It follows that $\mu_A^{X_A} = \restr{\nu_A}{C_0 (X_A)}$.

\begin{definition}
A $C_0(X)$-algebra $\A$ is said to be \emph{continuous} if the functions $x \mapsto \norm{a(x)}$ are continuous on $X$ for all $ a \in A$.
\end{definition}

\begin{remark}
If $\A$ is a continuous $C_0(X)$-algebra and $b \in \gb (A)$, the norm function $x \mapsto \norm{b(x)}$ is continuous on $X_A$.
\end{remark}

By Lee's Theorem~\cite[Theorem 5]{lee}, $\A$ is a continuous $C_0(X)$-algebra if and only if the base map $\phi_A$ is an open map.  In this case the local compactness of $\Pm(A)$~\cite[Corollary 3.3.8]{dix} passes to $X_A$.  It is evident that $\A$ is a continuous $C_0(X)$-algebra if and only if $\XA$ is a continuous $C_0(X_A)$-algebra.

Since we wish to identify $A$ with both $\go (\XA)$ and $\go (\A)$, it is natural to ask that $\gb (\XA)$ and $\gb (\A)$ be isomorphic also.  The following example shows how this could fail if we were to allow elements of $\gb (\A)$ to have domain $X$ rather than $X_A$.
\begin{example}
\label{e:notsurjective}
Let $A=C_0((0,1))$, $X=[0,1]$ and $\mu_A : C([0,1]) \rightarrow M (C_0(0,1))$ be given by pointwise multiplication, so that $\A$ is a $C([0,1])$-algebra.  Then $X_A=(0,1)$ and the fibres of $\A$ are given by $A_t= \mathbb{C}$ for $t \in X_A$ and $A_0=A_1= \{ 0 \}$.

Suppose that $b:X \to \coprod_{t \in X} A_t$ is continuous with respect to $\A$.  Then since $A_0=A_1 = \set{0}$, $b$ necessarily satisfies $b(0)=b(1)=0$. By continuity, it follows that for any $\eps >0$ and $t \in [0,1]$ there is $a \in A$ and $\delta>0$ with
\[
\norm{b(s)-a(s)} < \eps \mbox{ for all } s \in \brac{t-\delta,t+\delta.}
\]
Hence $b$ is continuous and vanishes at infinity on $(0,1)$, and so in fact $\gb (A) \cong C_0(0,1)$.

On the other hand, if we replace $X$ with $X_A$, we see that $\gb (A)$ consists of cross-sections $c: (0,1) \rightarrow \coprod_{t \in (0,1)} A_t$.  It is easily seen that in this case $\gb (A)$ is naturally isomorphic to $C^b ((0,1))$.

\end{example}

We now describe how a $C_0(X)$-algebra $A$ gives rise to a $C( \bX)$-algebra $M(A)$.

\begin{proposition}{\cite[Proposition 1.2]{arch_som_mult}}
\label{p:ma}
Let $\A$ be a $C_0 (X)$-algebra with base map $\phi_A : \Pm (A) \rightarrow X$.  Then the structure map $\mu_A$ extends to a unital $\ast$-homomorphism $\mu_{M(A)} : C ( \beta X) \rightarrow ZM(A)$, hence $\MA$ is a $C(\bX)$-algebra. 
\end{proposition}

\begin{remark}
When $X$ is compact, note that $\mu_{M(A)} = \mu_A$.  In particular this applies to the $C( \bXA)$-algebra $\Abxa$ of Definition~\ref{d:c0y}.
\end{remark}

It shall be convenient to fix some notation associated with $\A$ and $\MA$.

\begin{itemize}
\item By analogy with the ideals $J_x$ of $A$ defined in~(\ref{e:ix}), we define for $y \in \beta X$ the ideals $H_y$ of $M(A)$ via
\begin{eqnarray}
\label{e:hx} H_y & = & \mu_{M(A)} \left( \{ f \in C( \beta X : f(y) = 0 \} \right) M(A).
\end{eqnarray}

\item For $x \in X_A$, let $\pi_x :A \to A_x$ be the quotient $\ast$-homomorphism.  Then $\pi_x$ extends to a strictly continuous $\ast$-homomorphism $\tilde{\pi}_x : M(A) \to M(A_x)$ defined by
\begin{equation}
\label{e:pitilde}
\brac{\tilde{\pi}_x ( b )} \pi_x (a) = \pi_x (ba) \mbox{ and } \pi_x(a) \brac{ \tilde{\pi}_x (b) } = \pi_x (ab)
\end{equation} 
for all $a \in A$ and $b \in M(A)$~\cite{busby}.

\end{itemize}

Archbold and Somerset have studied the structure of the $C(\bX)$-algebra $\MA$ extensively~\cite{arch_som_ss},~\cite{arch_som_mult},~\cite{arch_som_ideals},~\cite{arch_som_inner}.  In particular, two important pathologies in the behaviour  of $\MA$ should be observed:
\begin{enumerate}
\item[(i)] It always holds that $\ker ( \tilde{\pi}_x ) \supseteq H_x$, but equality does not hold in general~\cite{arch_som_ss}.
\item[(ii)] If $\A$ is a continuous $C_0(X)$-algebra, it does not necessarily follow that $\MA$ is a continuous $C(\bX)$-algebra.  A complete characterisation of those $C_0(X)$-algebras $\A$, (with $A$ $\sigma$-unital) for which $\MA$ is continuous was obtained in~\cite[Theorem 3.8]{arch_som_mult}.
\end{enumerate}

\begin{lemma}
\label{l:hx}
Let $\A$ be a $C_0(X)$-algebra, let $x \in X$ and $m \in M(A)$ such that $\tilde{\pi}_x (m) = 0$.  Then $m \in H_x$ if and only if $\norm{\tilde{\pi}_y (m)} \rightarrow 0$ as $y \rightarrow x$.
\end{lemma}
\begin{proof}
By~\cite[Lemma 1.5(ii)]{arch_som_mult}, we have
\[
\norm{m + H_x } = \inf_W \sup_{y \in W} \norm{ \tilde{\pi}_y (m)},
\]
as $W$ ranges over all open neighbourhoods of $x$ in $X$, from which the result follows.
\end{proof}
Note that the fibres $M(A)_y$ of the $C( \bX)$-algebra $\MA$ are the quotient \Cst-algebras $\frac{M(A)}{H_y}$.  Hence for $m \in M(A)$ and $y \in Y$, $m(y)$ denotes $m + H_y$ (i.e. not $\tilde{\pi}_y (m)$ defined above).

\begin{example}
Let $\nhat = \mathbb{N} \cup \{ \infty \}$, and let $A = C( \nhat , c_0 ) \cong C_0 ( \nhat \times \mathbb{N})$ be the trivial continuous C$(\nhat )$-algebra with fibre $c_0$.  Then $\Pm (A)$ is canonically isomorphic to $\nhat \times \mathbb{N}$, and the base map $\phi_A : \Pm (A) \rightarrow \nhat$ is the projection onto the first coordinate.  Since $M(c_0 ) = \ell^{\infty} = C ( \beta \mathbb{N})$, it follows from~\cite[Corollary 3.4]{apt} that $M(A) = C( \nhat , \ell^{\infty}_{\beta} )$, where $\ell^{\infty}_{\beta}$ denotes $\ell^{\infty}$ with the strict topology induced by regarding $\ell^{\infty}$ as the multiplier algebra of $c_0$.  

$M(A)$, together  with the structure map $\mu_{M(A)} = \mu_A$, defines a $C(\nhat )$-algebra.  By~\cite[Corollary 3.9]{arch_som_mult}, $(M(A), \nhat , \mu_{M(A)})$ fails to be a continuous $C( \nhat )$-algebra.

For each $n \in \nhat$ let $\pi_n : A \rightarrow A_n \cong c_0$ be the quotient map, and $\tilde{\pi}_n : M(A) \rightarrow M(A_n) \cong \ell^{\infty}$ its strictly continuous extension to $M(A)$.  By~\cite[Theorem 4.9]{arch_som_ideals}, we have
\[ M(A)_n := M(A)/H_n = M(A_n)  \cong \ell^{\infty} \mbox{ for all }n \in \mathbb{N},\]
 while $H_{\infty} \subsetneq \ker ( \tilde{\pi}_{\infty})$ so that $M(A)_{\infty} \neq \ell^{\infty}$.  Moreover, by the same reference, there are uncountably many distinct norm-closed ideals $J$ of $M(A)$ satisfying $H_{\infty} \subsetneq J \subsetneq \ker (\tilde{\pi}_{\infty})$.  Note that by Lemma~\ref{l:hx} we have
 \[
 \ker ( \tilde{\pi}_{\infty} ) = \{ b \in C( \nhat , \ell^{\infty}_{\beta} ) : b( \infty ) = 0 \} 
 \]
 while
 \[
 H_{\infty} = \{ b \in C( \nhat , \ell^{\infty}_{\beta} ) : \norm{ b(n) } \rightarrow 0 \mbox{ as } n \rightarrow \infty \}.
 \]

\end{example}

\section{Morphisms and compactifications of $C_0(X)$-algebras}
\label{s:morph}
For a locally compact Hausdorff space $X$ and a pair of $C_0(X)$-algebras $\A$ and $(B,X,\mu_B)$, the natural notion of a morphism between $\A$ and $(B,X,\mu_B)$ is a $C_0(X)$-linear $\ast$-homomorphism. It is well-known that a $C_0(X)$-linear $\ast$-homomorphism $A \rightarrow B$ induces $\ast$-homomorphisms $A_x \rightarrow B_x$ of the corresponding fibre algebras.  In this section we describe how this notion can be extended to morphisms from a $C_0(X)$-algebra $\A$ to a $C_0(Y)$-algebra $\BY$, and we clarify what it means for such a morphism to be injective.

We introduce the notion of a compactification of a $C_0(X)$-algebra $\A$, generalising the classical definition for locally compact Hausdorff spaces.  Intuitively, we define a compactification of a $C_0(X)$-algebra as a $C(K)$-algebra $\B$, where $K$ is a compactification of $X_A$, such that $A$ is isomorphic to the ideal of continuous sections of $B$ that vanish at infinity on $X_A$.  At the level of \Cst -bundles, this is equivalent to $\B$ defining a bundle over $K$ whose restriction to $X_A$ coincides with that defined by $\A$.

The following definition is due to Kwasniewski~\cite{kwas}.
\begin{definition}
\label{d:morph}
Let $\A$ be a $C_0 (X)$-algebra and $\BY$ a $C_0 (Y)$-algebra.  By a \emph{morphism} $\A \rightarrow \BY$ we mean a pair $(\Psi ,\psi )$ where
\begin{enumerate}
\item[(i)] $\Psi :A \rightarrow B$ and $\psi : C_0 (X) \rightarrow C_0 (Y)$ are $\ast$-homomorphisms, and
\item[(ii)] for all $f \in C_0(X)$ and $a \in A$ we have
\[
\Psi ( \mu_A (f) a ) = \mu_B ( \psi (f) ) \Psi (a).
\]
\end{enumerate}
\end{definition}
\begin{remark}
Note that in the case that $Y=X$ and $\psi = \mathrm{Id}_{C_0(X)}$, a $\ast$-homomorphism $\Psi :A \to B$ gives rise to a morphism $(\Psi , \psi) : \A \to \BY$ if and only if $\Psi$ is $C_0(X)$-linear. 
\end{remark}

Let $(\Psi,\psi) : \A \to \BY$ be a morphism.  Note that dual to the $\ast$-homomorphism $\psi$ there is an open subset $U_{\psi}$ of $Y$ and a continuous, proper map $\psi^{\ast} : U_{\psi} \rightarrow X$,  such that $\psi$ is given by the formula
\begin{equation}
\label{e:delta}
\psi (f) (y) = \left\{ \begin{array}{rcl} 
f ( \psi^{\ast} (y) )& \mbox{ if } & y \in U_{\psi}, \\
0 & \mbox{ if} & y \not\in U_{\psi}
\end{array} \right.
\end{equation}
It is shown in~\cite[Section 3.1]{kwas} that the subspace $Y_B \subseteq Y$ (defined analogously to $X_A$ of~\eqref{e:xa}), is contained in $U_{\psi}$.  Moreover, composing $\Psi$ with the evaluation mappings, we get  $\ast$-homomorphisms $ \Psi_{y} : A_{\psi^* (y) } \rightarrow B_y  $ for all $y \in U_{\psi}$, so that $\Psi : A \to B$ satisfies
\begin{equation}
\label{e:morph}
\Psi (a) (y) = \left\{ \begin{array}{rcl}
\Psi_y ( a ( \psi^* (y) ) ) & \mbox{ if } & y \in U_{\psi} \\
0 & \mbox{ } &  y \not \in U_{\psi}
\end{array}
\right.
\end{equation}

Identifying  $A$ with $\go (A)$ and $B$ with $\go (B)$, $\Psi$ may be regarded as a $\ast$-homomorphism $\go (A) \to \go (B)$, and is completely determined by the formula
\[
\brac{\Psi (a)} (y) =
\Psi_y \brac{a(\psi^{\ast} (y))}
\]
for all $y \in Y_B$ ($\subseteq U_{\psi}$) and $a \in A$~\cite[Propositions 3.2 and 3.5]{kwas}.

\begin{remark}
\label{r:ext}
Let $(\Psi,\psi) : \A \to \BY$ be a morphism.  Then $\Psi$ extends to a morphism $\Psi^b : \gb (A) \to \gb (B) $ via
\[
\brac{\Psi^b (c)} (y) =
\Psi_y \brac{c(\psi^{\ast} (y))}
\]
for all $c \in \gb (A)$ and $y \in Y_B$.  Indeed, it is clear that $\Psi^b (c)$ is a cross section $Y_B \to \coprod_{y \in Y_B} B_y$ since $c(x) \in A_x$ for all $x \in X_A$.  To see that $\Psi^b (c)$ is continuous, let $y_0 \in Y_B$ and $\eps >0$. Then there is $a \in A$ and a neighbourhood $U$ of $\psi^* (y_0)$ in $X_A$ with
\[
\norm{a(x)-c(x)}< \eps
\]
for all $ x \in U$, and hence
\begin{align*}
\norm{\Psi(a)(y)-\Psi^b(c)(y)} & = \norm{\Psi_y \left[ a(\psi^* (y))-c(\psi^*(y)) \right] } \\
& \leq \norm{ a(\psi^* (y))-c(\psi^*(y))} < \eps
\end{align*}
for all $y \in (\psi^*)^{-1} (U)$.  Since $\Psi(a) \in B$, it follows that $\Psi^b(c)$ is indeed continuous.
\end{remark}

\begin{definition}
\label{d:inj}
A morphism $(\Psi, \psi) : \A \rightarrow \BY$ is called 
\begin{enumerate}
\item[(i)] \emph{injective} if the $\ast$-homomorphisms $\Psi$ and $\{ \Psi_{y} : y \in Y_A \}$ are all injective.
\item[(ii)] \emph{bijective} if it is both injective and $\Psi$ is a $\ast$-isomorphism
\item[(iii)] an \emph{isomorphism} if it is bijective and invertible, i.e. there exists a morphism $(\Phi , \phi ) : \BY \to \A$ such that $\Phi = \Psi^{-1}$ and
\[
\mu_A \brac{ ( \phi \circ \psi) (f) } = \mu_A (f) \mbox{ and } \mu_B \brac{(\psi \circ \phi )(g) } = \mu_B (g)
\]
for all $f \in C_0(X)$ and $g \in C_0 (Y)$.
\end{enumerate}
\end{definition}

The motivation for this definition of an injective morphism is as follows: consider $A \cong \go (A)$ and $B \cong \go (B)$ as algebras of cross-sections in the usual way.  Let $(\Psi, \psi) : \A \rightarrow \BY$ be an injective morphism, and identify $A \subseteq B$ and $A_{\psi (y)} \subseteq B_y$ as \Cst -subalgebras for all $y \in U_{\psi}$.  Then each $a \in A$ may be naturally regarded as a cross-section $a: Y \rightarrow \coprod_{y \in Y} B_y $ such that $a(y) \in A_{\psi (y)} \subseteq B_y$ for all $y \in U_{\psi}$ and $a(y)=0$ otherwise. In particular, we may identify $\go (A)$ canonically with a \Cst -subalgebra of $\go (B)$.

We remark that a bijective morphism need not be an isomorphism (i.e. invertible).  Clearly, if $(\Psi, \psi )$ is a bijective morphism and $\psi^{\ast}$ maps $Y$ homeomorphically onto $X$, the $(\Psi, \psi )$ will be an isomorphism. However, this requirement is too strong in general (since we do not require $Y_B$ to be dense in $Y$), this shall be addressed in Proposition~\ref{p:inj}(ii).

For our purposes, injective morphisms and isomorphisms arise in the manner described in the following proposition:

\begin{proposition}
\label{p:inj}
Let $(\Psi , \psi ) : \A \rightarrow \BY$ be a morphism.
\begin{enumerate}
\item[(i)] If $\Psi$ is an injective $\ast$-homomorphism and $\restr{\psi^*}{Y_B}$ a homeomorphism onto its image, then $(\Psi,\psi)$ is injective.
\item[(ii)] If $\Psi$ is a $\ast$-isomorphism and there is an open subset $V$ of $Y$ with $Y_B \subseteq V \subseteq U_{\psi}$ such that $\psi^{\ast}$ maps $V$ homeomorphically onto an open subset of $X$, then $(\Psi, \psi )$ is an isomorphism.
\end{enumerate}
In particular, if $X=Y$ and $\psi = \mathrm{Id}_{C_0(X)}$, then a $C_0(X)$-linear $\ast$-homomorphism $\Psi: A \to B$ defines and injective morphism (respectively, an isomorphism) $(\Psi,\psi): \A \to \BY$ if and only if $\Psi$ is injective (respectively, a $\ast$-isomorphism).
\end{proposition}
\begin{proof}
(i):  Let $y_0 \in Y_B$ and let $x_0 = \psi^* (y_0)$, and suppose for a contradiction that there is some $a \in A$ such that $\norm{a(x_0)} = 1$ and $ \Psi_{x_0} ( a(x_0)) = \Psi (a)(y_0) = 0$ in $B_{y_0}$.  Let $0 < \eps < \frac{1}{2}$, then by upper-semicontinuity of the function $ y \mapsto \norm{ \Psi (a) (y) }$ there is some open neighbourhood $W$ of $y_0$ in $Y$ such that $\norm{\Psi (a)(y)} < \eps $ for all $y \in W$.  
 
 By our assumption on $\psi^*$ the set $V:=\psi^*(W \cap Y_B)$ is a relatively open neighbourhood of $x_0$ in $\psi^{\ast} (Y_B)$. Let $\tilde{V}$ be an open subset of $X$ with $\tilde{V} \cap \psi^{\ast} (Y_B)=V$.

Take $f \in C_0 (X)$ such that $ 0 \leq f \leq 1$, $f(x_0)=1$ and such that $\restr{f}{X \backslash \tilde{V} } \equiv 0$. Note that 
\[
\norm{(\mu_A (f) a)(x_0)} = \norm{f(x_0)a(x_0)} = \norm{a (x_0)} = 1,
\]
and so
\[
\norm{\mu_A(f)a} = \sup_{x \in X} \norm{(\mu_A (f) a)(x)} \geq 1.
\]

Now, since $(\Psi,\psi)$ is a morphism, for all $y \in Y$ we have
\[
\Psi ( \mu_A(f) a )(y) = \left( \mu_B ( \psi (f)) \Psi (a) \right) (y) = f ( \psi^* (y)) \Psi (a)(y).
\]
Moreover, if $y \in Y \backslash W$, then 
\[ \psi^* (y) \in \psi^{\ast} (Y_B) \backslash V = \psi^{\ast} (Y_B) \backslash \tilde{V} \subseteq X \backslash \tilde{V}. \]
Hence 
\[ \Psi ( \mu_A (f) a) (y) = f \brac{ \psi^{\ast} (y)} \brac{\Psi (a)} (y)=0, \]
since $f$ vanishes on $X \backslash V$.  

On the other hand, if $y \in V$ then
\[
\norm{ \Psi ( \mu_A (f) a ) (y) } = \abs{ f ( \psi^* (y) ) }\ \norm{\Psi (a) (y)} \leq \norm{\Psi (a) (y) } < \eps.
\]
It follows that
\[
\norm{\Psi ( \mu_A (f) a ) } = \sup_{y \in Y} \norm{ \Psi ( \mu_A (f) a ) (y) } \leq \eps < \frac{1}{2},
\]
which contradicts our assumption that $\Psi$ is injective.

(ii): It is clear from part (i) that $(\Psi, \psi )$ is a bijective morphism.  Now, let $\phi^{\ast} : \psi (V) \to Y$ denote $\brac{ \restr{\psi^{\ast}}{V}}^{-1}$, and let $\phi : C_0 (Y) \to C_0 (X)$ be given by $\phi (g) = g \circ \phi^{\ast}$.  Then for $f \in C_0 (X)$ and $g \in C_0 (X)$ we have
\[
\brac{\phi \circ \psi}(f) = f \mbox{ and } \brac{ \psi \circ \phi } (g)  = g.
\]
Hence with $\Phi = \psi^{-1}$, $(\Phi, \phi )$ is clearly the inverse of $(\Psi, \psi )$.

The final assertion is then immediate.
\end{proof}

\begin{remark}
Note that if $(\Psi, \psi) : \A \rightarrow \BY$ is a morphism such that $\{ \Psi_y : y \in Y \}$ are all injective then it does not necessarily follow that $\Psi$ is injective.  For example let $X=[0,1) \cup (1,2]$, $Y=[0,1)$, $A = C_0 (X)$ and $B = C_0 (Y)$.  Let $\Psi = \psi : A \rightarrow B$ be the restriction homomorphism, then clearly $(\Psi, \psi)$ is a morphism. 

We have $U_{\psi} = [0,1)$ and $\psi^* : [0,1) \rightarrow [0,1) \cup (1,2]$ is the inclusion map.  Then for all $y \in U_{\psi}$, $\Psi_y : A_y \rightarrow B_y$ is the identity mapping, hence is injective, while clearly $\Psi$ is not injective.
\end{remark}
\begin{remark}
Injectivity of $\Psi$ and $\psi$ does not imply injectivity of the $\Psi_y$.  Let $C$ be a \Cst -algebra, $Y$ a compact Hausdorff space which contains more than one point, $B = C (Y,C)$ regarded as a $C(Y)$ algebra in the usual way.  If $X = \{ x \}$ is a one-point space, equip $A=B$ with the structure of a $C(X)$-algebra, and consider the morphism $(\Psi, \psi): \A \rightarrow \BY$ where $\Psi$ is the identity morphism and $\psi$ is the embedding of $\Cp$ as the constant functions on $Y$. $\Psi$ ans $\psi$ are injective. Then $\psi^*$ maps every point of $Y$ to $x$, and $\Psi_y : A=B \rightarrow B_y$ is the evaluation map for each $y$, which is not injective.
\end{remark}

\begin{definition}
\label{d:cfn}
Let $\A$ be a $C_0(X)$-algebra. A \emph{compactification} of $A$ is given by a $C(K)$-algebra $\B$ together with an injective $C(K)$-linear $\ast$-homomorphism $\Psi :A \to B $ such that
\begin{enumerate}
\item[(i)] $K$ is a compactification of $X_A$
\item[(ii)] $\Psi(A)$ is an essential ideal of $B$ and
\item[(iii)] The induced $\ast$-homomorphisms $\Psi_x : A_x \to B_x$ are isomorphisms for all $x \in X_A$.
\end{enumerate}
\end{definition}

When the $\ast$-homomorphism $\Psi$ is clear, we shall often say that $\B$ is a compactification of $\A$ and consider $A \subseteq B$.

Note that if the $C_0 (X)$-algebra $\A$ in Definition~\ref{d:cfn} is simply $C_0(X)$, then any compactification of $\A$ is given by $C(K)$ for some compactification $K$ of $X$.  More generally, if $A= C_0(X,B)$, then for any such $K$, $C(K,B)$ defines a compactification of $\A$, though in general there exist other compactifications in this case, see Corollary~\ref{c:triv}.

Based upon these elementary examples, one might conjecture that conditions (i) and (iii) of Definition~\ref{d:cfn} above imply condition (ii). However, the following example, from~\cite[p. 684]{kirch_wass} shows that this is not the case:
\begin{example}
Let $C$ be a non-nuclear \Cst -algebra. Then there is  a \Cst -algebra $D$ and a \Cst -subalgebra $D_0 \subseteq D$ for which the canonical map $ C \maxt D_0 \rightarrow C \maxt D$ is not injective.  Fixing $\alpha \in [0,1]$, let $X = [0,1] \backslash \{ \alpha \}$ and let $(A,X,\mu_A)$ be the trivial $C_0(X)$-algebra $C_0(X, C \maxt D)$.  

Let $D_{\alpha}$ be the continuous $C[0,1]$-algebra
\[ D_{\alpha} = \set{ f \in C([0,1],D) : f( \alpha ) \in D_0 }, \]
and let $B = C \maxt D_{\alpha}$.  Then $(B,[0,1],1_C \maxt \mu_{D_{\alpha}} )$ is a $C[0,1]$-algebra with fibres $B_t = C \maxt D$ for $t \neq \alpha$ and $B_{\alpha} = C \maxt D_0$ which is discontinuous at $\alpha$~\cite{kirch_wass}.  Moreover, it is easily seen that $(B,[0,1],1_C \maxt \mu_{D_{\alpha}} )$ is a compactification of $(A,X,\mu_A)$.

To see that $A$ is not an essential ideal of $B$, let $\pi : C \maxt D_0 \rightarrow C \maxt D$ be the $\ast$-homomorphism given by the universal property of the maximal tensor product. One can then construct an element $b \in B$ for which $b(x)=0$ for all $ x \in X$ and $\norm{b(\alpha)}=1$.  It is clear then that $ba$ and $ab$ are zero for all $a \in A$.
\end{example}

We also remark that if $\B$ is a compactification of $\A$, it is not true in general that there exists an injective morphism (or indeed any morphism) $(\Psi, \psi) : \A \to \B$ (since $X_A$ may fail to be locally compact).

The following proposition establishes a number of useful facts about compactifications that shall be used in subsequent sections. 
\begin{proposition}
\label{p:c0}
Let $\B$ be a compactification of $\A$ and $\Psi: A \to B$ the associated $C(K)$-linear $\ast$-homomorphism.  Then 
\begin{enumerate}
\item[(i)] Let $b \in B$, and identify $A_x$ with $B_x$ for each $x \in X_A$. If $b \in B$, the cross-section $X_A \rightarrow \coprod_{x \in X_A} A_x$, defined by the restriction $\restr{b}{X_A}$ of $b$ to $X_A$, belongs to $\gb(A)$.
\item[(ii)] For all $b \in B$ we have
\[
\norm{b} = \sup \set{ \norm{b(x)} : x \in X_A }.
\]
\item[(iii)] For all $y \in K \backslash X_A$, 
\[
\norm{b(y)} = \inf_W \sup_{x \in W \cap X_A} \norm{b(x)},
\]
where $W$ ranges over all neighbourhoods of $y$ in $K$.
\item[(iv)] $\Psi (A)$ may be identified with the ideal 
\[ \{ b \in B : b(y)=0 \mbox{ for all } y \in K \backslash X_A \} \]
of $B$.
\item[(v)] If in addition $X_A$ is locally compact, then we may identify $C_0(X_A)$ with an ideal of $C(K)$.  Under this identification, we have $\mu_B \brac{  C_0 (X_A) } \cdot B = \Psi (A)$.
\end{enumerate}
\end{proposition}
\begin{proof}
(i) Let $x \in X_A$ and $\eps > 0 $.  By definition of the isomorphism $\Psi_x : A_x \rightarrow B_x$ identifying $A_x$ with $B_x$, there is some $a \in A$ with $a(x) = b(x)$.  In particular, $\norm{(a-b)(x)}=0$.

Now, since $\B$ is a $C(K)$-algebra and $\Psi (a) \in B$, the norm function $y \mapsto \norm{\left( b - \Psi (a) \right) (y)}$ is upper-semicontinuous on $K$, and hence is upper-semicontinuous on $X_A$.  In particular, there is a neighbourhood $U$ of $x$ in $X_A$ such that $\norm{\left( b - \Psi (a) \right) (y)} < \eps$ for all $y \in U$.  Since $\Psi_y$ is injective for all $y \in X_A$, it follows that
\[
\norm{b(y)-a(y)}_{A_y} = \norm{ \left( b - \Psi (a) \right) (y) }_{B_y} < \eps
\]
for all $y \in U$.  Hence $\restr{b}{X_A} \in \gb (A)$.

(ii) Since $\Psi(A)$ is an essential ideal of $B$, there is an injective $\ast$-homomorphism $B \to M( \Psi (A) )$ which is the identity on $\Psi(A)$~\cite[Proposition 3.7(i) and (ii)]{busby}.  It then follows (using the construction of $M(\Psi (A))$ described in~\cite[Section 2]{busby}) that for all $b \in B$ we have
\[
\norm{b} = \sup \set{ \norm{ b \cdot \Psi (a) } : a \in A, \norm{a} \leq 1 }.
\]

Suppose for a contradiction that there were some $b \in B$ with $\norm{b}=1$ but that $\sup \set{ \norm{b(x)} : x \in X_A } = \alpha$ for some $0 \leq \alpha < 1$.  Since $\A$ is a $C_0(X)$-algebra and $\Psi$ is injective, we have 
\[
\norm{\Psi(a)} = \norm{a} = \sup_{x \in X_A}  \norm{a(x)}  = \sup_{x \in X_A}  \norm{\Psi (a)(x) } 
\]
for all $a \in A$.  In particular, if $\norm{a} \leq 1$ then for all $x \in X_A$ we would have
\begin{align*}
\norm{ \brac{b  \cdot \Psi (a)}(x)} &= \norm{b(x) \brac{\Psi (a)}(x) } \\
& \leq \alpha \cdot \norm{ \Psi (a) (x) } \leq \alpha \norm{a} < \norm{a}.
\end{align*}
Since $b \cdot \Psi (a) \in \Psi (a)$ for all $a \in A$, it would then follow that
\[
\norm{b \cdot \Psi (a) } = \sup_{x \in X_A} \norm{ \brac{ b \cdot \Psi (a)} (x)} < \norm{a}.
\]
In particular, this would imply that
\[
\norm{b} = \sup \set{ \norm{ b \cdot \Psi (a) } : a \in A, \norm{a} \leq 1 } \leq \alpha < 1,
\]
which is a contradiction.

(iii) Note that for $y \in K \backslash X_A$ condition (i) of Definition~\ref{d:cfn} ensures that we have $\Psi (a) (y) = 0$ for all $a \in A$.

Now let $b \in B$ such that $b(y)=0$ for all $y$ in $K \backslash X_A$, then for each $x \in X_A$ we have $b(x) \in \Psi_x \brac{A_x}$. Since $y \mapsto \norm{b(y)}$ is upper-semicontinuous on $K$, for each $\alpha > 0$ the set
\[
\{ y \in K : \norm{b(y)} < \alpha \}
\]
is open, and contains $K \backslash X_A$ by assumption.  Hence
\[
\set{ y \in K : \norm{b(y)} \geq \alpha }
\]
is a closed and hence compact subset of $K$, and moreover, is contained in $X_A$.  It follows that $\brac{ x \mapsto b(x) }$ defines an element of $\go (A)$, which by Theorem~\ref{t:go}(i), can only be the case if $b \in \Psi (A)$.

  (iv) Let $y \in K \backslash X_A$ and $\eps >0$. Since $x \mapsto \norm{b(x)}$ is upper-semicontinuous on $K$, there is an open neighbourhood $U$ of $y$ such that
  \[
  \sup_{x \in U} \norm{b(x)} \leq \norm{b(y)} + \eps.
  \]
  It follows that
  \[
  \inf_W \sup_{x \in W \cap X_A} \leq \norm{b(y)},
  \]
  as $W$ ranges over all neighbourhoods of $y$ in $K$.  
  
  Suppose that there were some open neighbourhood $U$ of $y$ in $K$ such that
  \[
  \sup_{x \in U \cap X_A} \norm{ b(x)} < \norm{ b(y) }.
  \]
Let $f \in C(K)$, $0 \leq f \leq 1$ with $f(y)=1$ and $\restr{f}{K \backslash U} \equiv 0$, and let $x \in X_A$. If $x \not\in U$, then $f(x)=0$ and hence $\norm{\brac{\mu_B(f) b} (x)}=0$.  If $x \in U$, then
\[
\norm{\brac{\mu_B(f) b} (x)} \leq \sup_{x \in U \cap X_A} \norm{b(x)} < \norm{b(y)}.
\]

Together with part (iii), this shows that
\begin{align*}
\norm{\mu_B (f) b } &= \sup_{x \in X_A} \norm{ \brac{ \mu_B (f) b } (x) } \\
& = \sup_{x \in U \cap X_A} \norm{b(x)} \\
& < \norm {b(y)} = \norm{ \brac{\mu_B(f) b} (y) },
\end{align*}
which is a contradiction.

(v)  Since $X_A$ is  locally compact and $K$ is a compactification of $X_A$, $X_A$ is also an open subset of $K$~\cite[3.15(d)]{gill_jer}. Then we may identify $C_0(X_A)$ with the ideal of $C(K)$ consisting of those $f \in C(K)$ vanishing on $K \backslash X_A$. In particular,
\[
\Psi \brac{ \mu_A^K (f) a } = \mu_B (f) \Psi (a)
\]
for all $f \in C_0 (X_A)$ and $a \in A$. 

Since $X_A$ is locally compact, we may regard $A$ as a $C_0(X_A)$ algebra with the same base map as $\A$ (restricting the range to $X_A$) as in Definition~\ref{d:c0y}. The corresponding structure map is then $\mu_A^{X_A}=\restr{\mu_A^K}{C_0 (X_A)}$, so that $\mu_A^K \brac{ C_0(X_A) } \cdot A = A$. Hence
\[
\Psi (A) = \mu_B (C_0 (X_A)) \cdot \Psi (A) \subseteq \mu_B (C_0 (X_A)) \cdot B,
\]
so it remains to show the reverse inclusion.

Letting $B^0 = \mu_B (C_0 (X_A)) \cdot B$, we see that $(B^0,X_A,\restr{\mu_B}{C_0 (X_A)})$ is a $C_0(X_A)$-algebra.  Moreover, if $b \in B$ and $x \in X_A$, choosing $f \in C_0(X_A)$ with $f(x)=1$ we have $\mu_B (f) b \in B^0$ and $\left( \mu_B(f)b \right)(x) = b(x)$ .  Hence for all $x \in X_A$ the fibre algebras $B_x^0$ are equal to $B_x $, which are in turn isomorphic to $A_x$ (since $\B$ is a compactification of $\A$).

In particular, for every $x \in X_A$ and $c \in B^0_x$ there is $a \in A$ with $a(x)=c$.  Hence by Theorem~\ref{t:go}(iii), $\Psi (A) = B^0$, which completes the proof.
\end{proof}

\section{The Stone-\v{C}ech compactification of a $C_0(X)$-algebra}
\label{s:cfn}

In this section we show that any $C_0(X)$-algebra $\A$ gives rise to a \Cst -algebra $\Ab$, with $A \subseteq \Ab \subseteq M(A)$, such that $\Abx$ is a $C( \bX)$-algebra, and defines a compactification of $\A$ in a natural way.

In general, $\Ab$ will be strictly contained in $M(A)$ (since we require the fibre algebras at points of $X$ to coincide with those of $A$).  For example, when $A$ is a trivial $C_0(X)$-algebra of the form $A = C_0 (X,B)$, then we shall see that $\Ab = C^b (X,B)$, while $M(A) = C^b(X , M(B)_{\beta} )$.

\begin{definition}
Let $\A$ be a $C_0(X)$-algebra, and define the closed ideal $\Ab$ of $M(A)$ via
\begin{equation}
\label{e:ab}
\Ab = \bigcap \set{ A + H_x : x \in X_A }.
\end{equation}
\end{definition}
Note that $\Ab$ is an essential ideal of $M(A)$ since it contains $A$.  We shall show that $\Ab$ gives rise naturally to a compactification of $\A$, which has many properties analogous to those of the Stone-\v{C}ech compactification of a locally compact Hausdorff space.

It is important to note that the definition of $\Ab$ is independent of the base space $X$ containing $X_A$.  More precisely, replacing $\A$ with the $C_0(X)$-algebra $\AY$ of Definition~\ref{d:c0y}, the ideals $H_x$ for $x \in X_A$ are unchanged~\cite[Lemma 2.4]{arch_som_ideals}, hence the same is true for $\Ab$.

Part (ii) of Proposition~\ref{p:abiff} was established in~\cite[Theorem 3.3]{arch_som_inner}.

\begin{proposition}
\label{p:abiff}
For $C_0(X)$-algebra $\A$ we have
\begin{enumerate}
\item[(i)] $\Ab = M(A)$ if and only if $\mu_A ( C_0(X)) \cap A \not\subseteq J_x$ for any $x \in X_A$.
\item[(ii)] If $X_A$ is compact, then $\Ab = A$. If $A$ is $\sigma$-unital, then $\Ab = A$ if and only if $X_A$ is compact.
\end{enumerate}
\end{proposition}

\begin{proof}
 To see (ii), suppose first that $\Ab = M(A)$, so that in particular, $A+H_x = M(A)$ for all $x \in X_A$.  If there were some $x \in X_A$ with $\mu_A(C_0(X)) \cap A \subseteq J_x$, then by~\cite[Lemma 2.1 (i)$\Rightarrow$(iii)]{arch_som_mult}, there would be some $R \in \Pm (M(A))$ with $R \supseteq A$ and $\phi_{M(A)} (R) = x$.  Hence (using~(\ref{e:hx})) $R \supseteq H_x$, so that $R \supseteq A+H_x$.  This would imply that $A+H_x \subsetneq M(A)$, which is a contradiction.

Conversely, suppose that $\mu_A ( C_0(X)) \cap A \not\subseteq J_x$ for any $x \in X_A$.  Then we have that the canonical embedding of $A_x$ into $M(A)_x$ is surjective $x \in X_A$ by~\cite[Proposition 2.2(iii)]{arch_som_mult}.  But then for any $m \in M(A)$, there is $a \in A$ with $m(x)=a(x)$, and hence $m-a \in H_x$.  It follows that $m \in A+H_x$, so $A+H_x = M(A)$.  Since this is true for all $x \in X_A$, $\Ab = M(A)$.

(i) is shown in~\cite[Theorem 3.3]{arch_som_inner} (note that the proof of (ii)$\Rightarrow$(i) therein does not require $A$ to be $\sigma$-unital). 
\end{proof}

\begin{example}
Let $B$ be a \Cst -algebra and let $\A$ be the trivial $C_0(X)$-algebra $A=C_0(X,B)$.  Then by~\cite[Corollary 3.4]{apt}, $M(A) = C^b ( X , M(B)_{\beta} )$, where $M(B)_{\beta}$ denotes $M(B)$ equipped with the strict topology.

For each $x \in X$, the $\ast$-homomorphism $\tilde{\pi}_x : M(A) \rightarrow M(B)$ of~\eqref{e:pitilde} coincides with the evaluation map at $x$. Then by Lemma~\ref{l:hx}, $H_x$ is the ideal
\[
\{ c \in M(A) : \norm{\tilde{\pi}_y(c)} \rightarrow 0 \mbox{ as } y \rightarrow x \}
\]
of $M(A)$.  Hence for $m \in M(A)$, $m$ belongs to $A+H_x$ if and only if there is some $a \in A$ such that $m-a \in H_x$, i.e., if and only if there is some $b \in B$ such that
\[
\norm{m(y)-b} \rightarrow 0 \mbox{ as } y \rightarrow x.
\]
It follows that $\Ab$ consists of the subalgebra $C^b (X,B)$ of $M(A)$.

Note that in this example, we have $X_A = X$, and hence the results of Proposition~\ref{p:abiff} are easily observed.  Indeed, it is clear that $A = \Ab$ if and only if $C_0(X,B) = C^b (X,B)$, which occurs if and only if $X$ is compact.

Moreover, the structure map $\mu_A :C_0(X) \rightarrow ZM(A)$ in this case is given by pointwise multiplication by elements of $C_0(X)$.  Hence if $B$ is unital,
\[ \mu_A (C_0(X)) \cap A =  \set{ f \cdot 1_B : f \in C_0(X) } \not\subseteq J_x
\]
for any $x \in X$, while if $B$ is non-unital, $\mu_A(C_0(X)) \cap A = \emptyset$.

It is clear the $\Ab = M(A)$ if and only if $\Ab$ is unital, which occurs if and only if $B$ is, hence if and only if $\mu_A (C_0(X)) \cap A \not\subseteq J_x$ for any $x \in X_A = X$.
\end{example}

Let $\A$ be a $C_0(X)$-algebra and $K$ a compactification of $X_A$.  Then for any $f \in C(K)$, $\restr{f}{X_A}$ belongs to $C^b(X_A)$.  It follows that $\restr{f}{X_A}$ extends to a continuous function $\overline{\restr{f}{X_A}} \in C( \bXA)$.  In particular, we get a unital, injective $\ast$-homomorphism $C(K) \to C( \bXA)$, and so we may identify $C(K)$ with a unital \Cst -subalgebra of $C(\bXA)$.

In the following Theorem we shall make use of the $C(\bXA)$-algebra $\Abxa$ of Definition~\ref{d:c0y}.

\begin{theorem}
\label{t:ab}
Let $\A$ be a $C_0(X)$-algebra and $\Abxa$ the $C(\bXA)$-algebra canonically defined by $\A$.  Then
\begin{enumerate}
\item[(i)] There is an injective $\ast$-homomorphism $\pi: ZM(\Ab) \rightarrow ZM(A)$, and hence letting $\mu_{\Ab} : C( \beta X_A) \rightarrow ZM( \Ab )$ be the composition $\mu_{\Ab} = \pi \circ \nu_A$, the triple $(\Ab, \bX , \mu_{\Ab}))$ is a $C( \bX)$-algebra. Moreover, if $K$ is any compactification of $X_A$, then letting $\mu_{\Ab}^K = \restr{\mu_{\Ab}}{C(K)}$, $(\Ab, K , \mu_{\Ab}^K )$ is a $C(K)$-algebra.
\item[(ii)] For any compactification $K$ of $X_A$, $\AbK$ is a compactification of $\A$.
\item[(iii)] If $\A$ is a continuous $C_0(X)$-algebra such that $X_A=X$, then $\Abx$ is a continuous $C(\bX)$-algebra.
\end{enumerate}
\end{theorem}
\begin{proof}
(i): Since $\Ab$ is an essential ideal of $M(A)$, there is an injective $\ast$-homomorphism $\iota: M(A) \rightarrow M(\Ab)$ which is the identity on $\Ab$. Setting $\pi= \restr{\iota}{ZM(A)}$, then since $A$ is essential in $\Ab$, $\pi$ maps $ZM(A)$ into $ZM(\Ab)$. It follows that there is a unital $\ast$-homomorphism $\mu_{\Ab} := \pi \circ \nu_A : C( \bX ) \rightarrow ZM(\Ab)$, so that $\Abx$ is a $C(\beta X_A)$-algebra.

If $K$ is a compactification of $X_A$, then since $C(K)$ is a unital \Cst -subalgebra of $C( \beta X_A)$, the $\ast$-homomorphism $\mu_{\Ab}^K : C(K) \rightarrow ZM( \Ab )$ is unital, and in particular, non-degenerate.

(ii): It is clear from the definition of $\mu_{\Ab}^K$ that we have $\mu_{\Ab}^K (f) a = \mu_A^K(f)a$ for all $f\in C(K)$ and $a \in A$, hence $A \to \Ab$ is $C(K)$-linear and we have a morphism $\AK \to \AbK$. It is an injective morphism (i.e. the maps $A_x \rightarrow \Ab_x$ are injective) by Proposition~\ref{p:inj}. Hence conditions (i) and (ii) of Definition~\ref{d:cfn} are satisfied.

 To see condition (iii) of Definition~\ref{d:cfn}, we first claim that the fibre algebras $\Ab_x$ of $\AbK$, where $x \in X_A$, do not depend on our choice of compactification $K$ of $X_A$.  For clarity, let us denote by $\psi_K : C(K) \rightarrow C( \beta X_A)$ the usual injective $\ast$-homomorphism, and by $\psi_K^{\ast} : \beta X_A \rightarrow K$ the dual continuous surjection.  Then $\restr{\psi_K^*}{X_A}$ is the identity map, and moreover, $\psi_K^* ( \beta X_A \backslash X_A ) = K \backslash X_A$ by~\cite[Theorem 6.12]{gill_jer}.  
 
 Denote by $\phi_{\Ab}^K : \Pm ( \Ab ) \rightarrow K$ the base map of the $C(K)$-algebra $\AbK$.  Then it is clear from the definition of $\mu_{\Ab}^K  : C(K) \rightarrow ZM ( \Ab )$ above that
 \[
 \phi_{\Ab}^K = \psi^*_K \circ \phi_{\Ab}^{\beta X_A}.
 \]
 In particular, this implies that for $P \in \Pm ( \Ab )$ and $x \in X_A$
 \[
 \phi_{\Ab}^K (P) = x \mbox{ if and only if } \phi_{\Ab}^{\bX} (P) = x.
 \]
 Hence for all such $x$, we have
 \[
 \mu_{\Ab}^K \left( C_0 ( K \backslash \{ x \} ) \right) \cdot \Ab = \mu_{\Ab}^{\bX} \left( C_0 ( \bX \backslash \{ x \} ) \right) \cdot \Ab.
 \]
 Since the fibres corresponding to points of $X_A$ of $\AbK$ and $\Abx$ are quotients of $\Ab$ by these ideals, it follows that $\Ab_x$ does not depend on the choice of compactification  $K$ of $X_A$.

 Now let $x \in X_A$ and note that for $c \in \Ab$, $c(x)=0$ (as a section of $\AbK$) whenever $c \in H_x$ by the definition of $\mu_{\Ab}^K$. Now if $b \in \Ab$, write $b = a + c$ for some $a \in A$ and $c \in H_x$, then it is clear that $a(x)=(a+c)(x)=b(x)$.  In particular, the canonical map $A_x \rightarrow \Ab_x$ is surjective for all $x \in X_A$.

(iii): We first show that norm functions of elements of $\Ab$ are continuous on $X_A$.  Note that $X_A$ is locally compact since $\A$ is continuous.

Take $b \in \Ab$ and let $f_b: X \rightarrow \Real$ be the function $f_b (x) = \norm{ b (x) }$. For $x_0 \in X$ let $K_{x_0}$ be a compact neighbourhood of $x_0$ in $X$ and $g \in C_0 (X)$ with $\restr{g}{K_{x_0}} \equiv 1$.  Then by Proposition~\ref{p:c0}(v), $\mu_{\Ab}(g)b \in A$, and $x \mapsto \norm{(\mu_{\Ab}(g) b)(x)}$ is continuous on $X$.  Since $x \mapsto \norm{(\mu_{\Ab}(g) b)(x)}$ agrees with $f_b$ on $K_{x_0}$, it follows that $f_b$ is continuous at $x_0$ for any $x_0$, and so $f_b$ is continuous on $X$.

 Denote by $\overline{f_b}$ the unique extension of $f_b$ to a continuous function on $\beta X$, and by $g_b$ the function $g_b(y) = \norm{b(y)}$ on $\beta X$.  We claim that $\overline{f_b} = g_b.$
 
 Let $y \in \bX$ and let $(x_{\alpha})$ be a net in $X$ converging to $y$.  Note that for each $\alpha$ the function $b \mapsto \norm{ b (x_{\alpha})}$, where $b \in \Ab$, is a C$^{\ast}$-seminorm on $\Ab$.  It follows that $b \mapsto \overline{f_b} (y) = \lim_{\alpha} f_b (x_{\alpha})$ is a C$^{\ast}$-seminorm on $\Ab$.  Since the C$^{\ast}$-norm on $\Ab_y$ is unique, it suffices to prove that $\overline{f_b }(y) = 0  $ if and only if $b(y) = 0 $.
 
 Suppose first that $b (y) = 0$ and $\varepsilon >0$.  Then since $x \mapsto \norm{b (x)}$ is upper-semicontinuous on $\bX$, there is an open neighbourhood $U$ of $y$ in $\bX$ such that $\norm{b (x) } < \varepsilon $ for all $x \in U$.  It follows that there is $\alpha_0$ with $f_{b} (x_{\alpha}) = \norm{b (x_{\alpha})} < \varepsilon$ whenever $\alpha \geq \alpha_0$.  Since $\overline{f_{b}}$ is continuous and $x_{\alpha} \rightarrow y$, this implies that $\overline{f_b} (y) = 0$.
 
 Now suppose that $\overline{f_b}(y) = 0$, and let $\varepsilon > 0 $.  Let $U = \{ x \in \bX : \overline{f_b} (x) < \varepsilon \}$ and take $g \in C( \bX )$ with $0 \leq g \leq 1$, $g(y)=1$ and $\restr{g}{\bX \backslash U } \equiv 0$.  Then
\begin{align*}
\norm{ b - \mu_{\Ab}(1-g) \cdot b } & = \norm{ \mu_{\Ab}(g) \cdot b } \\
& = \sup_{x \in \bX } \norm{ (\mu_{\Ab}(g) \cdot b ) (x) } \\
& = \sup_{x \in X} \norm{( \mu_{\Ab}(g) \cdot b ) (x) } \\
& = \sup_{x \in U \cap X} | g(x) | \norm{b (x) } \leq \varepsilon. 
\end{align*}
Since $\varepsilon$ was arbitrary, it follows that $b \in H_y$.  Hence $b (y) = 0 $ by the definition of $\Ab_y$.

\end{proof}

\begin{example}
If $\A$ is as in Example~\ref{e:ab} and $K$ is a compactification of $X$ then $\AbK$ is a compactification of $\A$. Here $\mu_{\Ab}^K$ is given by pointwise multiplication.  

It is easy to see that $\AbK$ need not be a continuous $C(K)$-algebra in general. Indeed, consider $X = \mathbb{R}$, $A=C_0(\mathbb{R})$ and $K$ the one-point compactification of $\mathbb{R}$.
\end{example}

\section{The algebra of bounded continuous sections}
\label{s:ext}

In this section, we examine the structure of the algebra of continuous sections of $\AbK$, and in particular, of $\Abx$.  We show that $\Ab$ has the following Stone-\v{C}ech-type property: every continuous bounded section in $\gb (A)$ has a unique extension to a continuous section in $\Gamma^b( \Ab)$ (irrespective of our choice of compactification $K$ of $X_A$ over which $\Ab$ defines a $C(K)$-algebra). As a consequence, we show that every trivial C$^{\ast}$-bundle over a locally compact Hausdorff space $X$ extends uniquely to a continuous \Cst -bundle over $\bX$ (though this extension may fail to be trivial) with the Stone-\v{C}ech extension property above.

For Theorem~\ref{t:ext} below, it will be necessary to return to the full notation of Definition~\ref{d:cts} to avoid ambiguity.  Let $\A$ be a $C_0(X)$-algebra and $\B$ a compactification of $\A$.  We shall denote by $\Gamma^b (\B)$ the \Cst-algebra of norm-bounded cross-sections $K \rightarrow \coprod_{y \in K} B_y$ that are continuous with respect to the $C(K)$-algebra $\B$ (in the sense of Definition~\ref{d:cts}). 

 Note that this notation allows us to distinguish between section algebras $\Gamma^b ( \AbK )$ for different compactifications $K$ of $X$, where $\AbK$ is  defined as in Theorem~\ref{t:ab}.  We shall continue to use $\gb (A)$ for bounded continuous cross sections $X_A \to \coprod_{x \in X_A} A_x$, since this is unambiguous.

\begin{theorem}
\label{t:ext}
Let $\B$ be a compactification of the $C_0(X)$-algebra $\A$, and identify each $b \in B$ with the corresponding element of $\go ( \B )$. 
\begin{enumerate}
\item[(i)] The map $b \mapsto \restr{b}{X_A}$ is an injective $\ast$-homomorphism $\Gamma_0 (\B) \rightarrow \gb (A)$.
\item[(ii)]  If $B= \Ab$, then the map $b \mapsto \restr{b}{X_A}$ is a $\ast$-isomorphism of $\Gamma_0 (\AbK)$ onto $\gb( A )$.
\item[(iii)]  Every continuous section $a \in \gb (A)$ extends uniquely to a continuous section $\overline{a} \in \Gamma^b ( \AbK)$.
\end{enumerate}
\end{theorem}
\begin{proof}
(i): The fact that the restriction of any $b \in B$ to $X_A$ belongs to $\gb (A)$ follows from Proposition~\ref{p:c0}(i).  Moreover, it is clear that the map $b \mapsto \restr{b}{X_A}$ is a $\ast$-homomorphism.  Injectivity follows from Proposition~\ref{p:c0}(ii), since if $b,c \in B$ with $\restr{b}{X_A} = \restr{c}{X_A}$, then $\norm{b-c} = \sup_{x \in X_A} \norm{ (b-c)(x) } = 0$.

(ii): Now let $B=\Ab$, and regard both $\Ab$ and $\gb (A)$ as \Cst -subalgebras of $M(A)$. It is clear from part (i) that $\Ab \subseteq \gb (A)$. Suppose that $m \in \gb (A)$ and $x \in X_A$, then there is some $a \in A$ with $a(x)=m(x)$.  Then since $a-m \in \gb (A)$ it follows that $y \mapsto \norm{(a-m)(y)}$ is upper-semicontinuous on $X_A$. As $(a-m)(x)=0$,  there is some open neighbourhood $U$ of $x$ in $X_A$ with $\norm{(a-m)(y)} < \varepsilon $ for all $y \in U$.  

Denote by $\tilde{\pi}_y : M(A) \rightarrow M(A_x)$ the $\ast$-homomorphism extending $\pi_y : A \rightarrow A_y$ of equation~(\ref{e:pitilde}) for each $y \in X_A$.  Then since $\norm{\tilde{\pi}_y (b) } \leq \norm{ b(y) }$ for all $y \in X_A$ and $b \in M(A)$, we have 
\[ \norm{\tilde{\pi}_y ( a-m )} \leq \norm{ (a-m)(y)} < \eps \]
for all $y \in U$.  Hence $\norm{\tilde{\pi}_y (a-m) } \rightarrow 0$ as $ y \rightarrow x$, so that $a-m \in H_x$ by Lemma~\ref{l:hx}.  

This ensures that for any $m \in \gb (A)$ and $x \in X_A$, $m \in A + H_x$, so that $ \gb (A) \subseteq \Ab$.

(iii): If $a \in \gb (A)$, we choose $\overline{a}$ to be the element of $\Ab$ whose image under the isomorphism of part (ii) is $a$. Then $\overline{a}$ canonically defines an element of $\Gamma^b ( \AbK)$ with $\overline{a}(x) = a (x)$ for all $x \in X_A$. Uniqueness follows from the fact that the restriction map of part (ii) is an isomorphism.
\end{proof}

The conclusion of Theorem~\ref{t:ext} might appear counter-intuitive at first, in that the extension property of $\AbK$ is true for \emph{any} compactification $K$ of $X$.  Indeed, for a commutative \Cst -algebra $C_0(X)$, we know that $\bX$ is the \emph{unique} compactification of $X$ having the property that every $f \in C^b (X)$ extends to a continuous function $\overline{f} \in C( \bX)$.  Nonetheless, given any compactification $K$ of $X$, $C( \bX )$ may be equipped with the structure of a $C(K)$-algebra. The unique extension $\overline{f} \in C( \bX)$ of $f \in C^b (X)$ above then defines a continuous section $K \rightarrow \coprod_{y \in K} C( \bX)_y$, as the following example describes:

\begin{example}
Let $\A$ be the $C_0 ( X )$-algebra defined by $C_0 ( X )$.  Now, $\Ab \cong C^b (X) \cong C( \bX)$, and by Theorem~\ref{t:ab}(i), the $C(K)$-algebra $\AbK$ gives rise to a compactification of $\A$ for any compactification $K$ of $X$.  Note that $\mu_{\Ab}^K : C(K) \rightarrow ZM( \Ab)$ is given by the natural (unital) embedding of $C(K)$ into $C( \bX)$ in each case. Hence the corresponding base map $\phi_{\Ab}^K : \Pm ( \Ab ) \cong \bX \rightarrow K$ is the Stone-\v{C}ech extension to $\bX$ of the homeomorphic embedding of $X$ into $K$.

Since $(\Ab, \bX , \mu_{\Ab} )$ is a compactification of $\A$, the fibre algebras are given by $\Ab_x = A_x = \mathbb{C}$ for all $x \in X$. For $y \in K \backslash X$ we have $\Pm (\Ab_y ) \cong (\phi_{\Ab}^K )^{-1} (y)$, so that
\[
\Ab_y = C ( (\phi_{\Ab}^K )^{-1} (y) )
\]
for all such $y$.

Note that in the particular case where $K = \hat{X} := X \cup \{ \infty \}$ (the one-point compactification of $X$), we have
\[
\Ab_{\infty} = C( \bX \backslash X ).
\]
\end{example}
\begin{remark}
The fact that the fibre algebras of $\AbK$ at points of $x$ are given by $\mathbb{C}$ may be deduced from the fact that $\phi_{\Ab}^K : \bX \rightarrow K$ maps $\bX \backslash X$ to $K \backslash X$~\cite[Theorem 6.12]{gill_jer}.  Indeed, it then follows that $(\phi_{\Ab}^K)^{-1} ( \{ x \} ) = \{ x \}$ for any $x \in X$, and so $\Pm ( \Ab_x ) = \{ x \}$ consists of a single point for all such $x$.
\end{remark}
\begin{remark}
\label{r:unique}
The fact that $\bX$ is the unique compactification of $X$ with the property that any $f \in C^b (X)$ extends to $\overline{f} \in C( \bX)$ can still be recovered in the language of  $C_0(X)$-algebra compactifications.  Indeed, let the continuous $C_0(X)$-algebra $\A$ be given by $A=C_0(X)$ in the usual way.  Then since $\Ab = C^b (X)$, every positive function $g \in C^b (X)_+$ occurs as the norm-function $x \mapsto \norm{g(x)}$.  

Suppose that $K$ is a compactification of $X$ such that $\AbK$ is a continuous $C(K)$-algebra, every such $g$ extends uniquely to a continuous function $\overline{g} \in C(K)$, namely, the norm function of the continuous section $\overline{g} \in \Gamma ( \AbK )$ of Theorem~\ref{t:ext}.  Since this is true for all $g \in C^b (X)_+$, it follows that $K = \bX$.

In summary, $(\Ab, \bX , \mu_{\Ab})$ is unique in the following sense: let $A=C_0(X)$ and $\A$ the corresponding continuous $C_0(X)$-algebra. Let $\B$ be a compactification of $\A$ such that
\begin{enumerate}
\item[(i)] $\B$ is a continuous $C(K)$-algebra, and
\item[(ii)] Every $b \in \gb (\A)$ extends to $\overline{b} \in \Gamma ( \B )$.
\end{enumerate}
Then there is an isomorphism $\B \cong (\Ab , \bX , \mu_{\Ab} )$.  We shall extend this to more general continuous $C_0(X)$-algebras in Section~\ref{s:univ}.
\end{remark}

The Stone-\v{C}ech compactification of a completely regular space $X$ is functorial in the sense that if $Y$ is another completely regular space and $\phi : X \to Y$ a continuous map, $\phi$ has a unique extension to $\phi^{\beta} : \bX \to \beta Y$.  As a consequence of Theorem~\ref{t:ext}, we can now show that $\Abx$ has the corresponding property with respect to morphisms of $C_0(X)$-algebras.

\begin{corollary}
\label{c:functorial}
Let $\A$ be a $C_0(X)$-algebra, $\BY$ a $C_0(Y)$-algebra and $(\Psi,\psi): \A \to \BY$ a morphism.  Then there is a morphism
\[
\left( \Psi^{\beta} , \psi^{\beta} \right) : \Abx \to ( B^{\beta} , \beta Y_B , \mu_{B^{\beta}} )
\]
where $\Psi^{\beta}$ extends $\Psi$ and $\restr{\psi^{\beta} (f)}{Y_B} = \restr{\psi (f)}{Y_B}$ for all $f \in C_0 (X)$.
\end{corollary}
\begin{proof}
As discussed in Remark~\ref{r:ext}, we have a $\ast$-homomorphism $\Psi^b : \gb (A) \to \gb (B)$ defined via
\[
\Psi^b(c)(y) =  \Psi_{y} (c(\psi^*(y)))
\]
for all $c \in \gb (A)$ and $y \in Y_B$.  Taking the composition with the isomorphisms $\Ab \cong \gb (A)$ and $B^{\beta} \cong \gb (B)$ of Theorem~\ref{t:ext} gives a $\ast$-homomorphism $\Psi^{\beta} : \Ab \to B^{\beta}$ extending $\Psi$.

To construct $\psi^{\beta}$ we first define a $\ast$-homomorphism $\psi^b : C^b (X_A) \to C^b (Y_B)$ via the composition
\[
\psi^b (f)(y) = f ( \psi^* (y) ),
\]
where $f \in C^b (X_A)$ and $y \in Y_B$ (note that the domain of $\psi^*$ contains $Y_B$).  Then $\psi^b$ induces a $\ast$-homomorphism $\psi^{\beta} : C( \beta X_A ) \to C( \beta Y_B)$, which has the property that $\restr{\psi^{\beta} (f)}{Y_B} = \restr{\psi (f)}{Y_B}$ for all $f \in C_0 (X)$ by construction.

Finally, to see that $(\Psi^{\beta},\psi^{\beta} )$ is indeed a morphism, first note that for any $c \in \Ab , f \in C( \beta X_A )$ and $y \in Y_B$, the definitions of $\Psi^{\beta}$ and $\psi^{\beta}$ ensure that
\begin{align*}
\left[ \Psi^{\beta} \left( \mu_{\Ab} (g) c \right) \right] (y) & = \Psi_y \left[ g \left( \psi^* (y) \right) c \left( \psi^* (y) \right) \right] \\
& = \brac{\psi^{\beta} (g) (y)} \brac{ \Psi^{\beta} (c) (y) }\\
& = \left[ \mu_{B^{\beta}} \left( \psi^{\beta}(g) \right) \Psi^{\beta} (c) \right] (y).
\end{align*}
 In other words,
\[
 \left[\Psi^{\beta} \left( \mu_{\Ab} (g) c \right) - \mu_{B^{\beta}} \left( \psi^{\beta}(g) \right) \Psi^{\beta} (c) \right] (y) =0
\]
for $y \in Y_B$. Since $(B^{\beta}, \beta Y , \mu_{B^{\beta}} )$ is a compactifciation of $\BY$, Proposition~\ref{p:c0}(ii) shows that any $d \in \B^{\beta}$ has $\norm{d} = \sup_{y \in Y_B} \norm{d(y)}$, so we must have
\[
\Psi^{\beta} \left( \mu_{\Ab} (g) c \right) = \mu_{B^{\beta}} \left( \psi^{\beta}(g) \right) \Psi^{\beta} (c), 
\] 
which shows that $(\Psi^{\beta},\psi^{\beta})$ is indeed a morphism.
\end{proof}

Note that the commutative \Cst -algebra $A = C_0(X)$ is a trivial $C_0(X)$-algebra with fibre $\mathbb{C}$, and that $\Ab = C^b (X)$ is a trivial $C( \bX)$-algebra with fibre $\mathbb{C}$.  It is natural to ask whether or not the same is true for a trivial $C_0(X)$-algebra $A$ of the form $A= C_0 (X,B)$ for some \Cst -algebra $B$.  It was shown in~\cite{williams_tensor} that it is not true in general that every $f \in C^b (X,B)$ extends to a continuous function $f: \bX \rightarrow B$.  In particular, we cannot expect $C^b(X,B)$ and $C( \bX , B )$ to be isomorphic in general.

Consider now the usual identification of $f \in C^b (X,B)$ with the corresponding cross section $X \rightarrow \coprod_{x \in X} B$ of the trivial bundle over $X$ with fibre $B$.  Corollary~\ref{c:triv} below shows that Theorems~\ref{t:ab} and~\ref{t:ext} give rise to an extension $\overline{f} : \bX \rightarrow \coprod_{y \in \bX} \Ab_y$ of $f$ to a continuous section of the $C( \bX)$-algebra $(\Ab , \bX , \mu_{\Ab} )$.

Recall that a locally compact Hausdorff space $X$ is said to be \emph{pseudocompact} if every continuous function $f : X \rightarrow \mathbb{C}$ is necessarily bounded.  

\begin{corollary}
\label{c:triv}
Let $B$ be a \Cst -algebra and let $\A$ be the trivial $C_0(X)$-algebra defined by $A = C_0 (X,B)$.  Then
\begin{enumerate}
\item[(i)] $\gb (A) = C^b (X, B)$, hence
\item[(ii)] every $a \in C^b (X,B)$ (regarded as a cross-section of the trivial bundle over $X$ with fibre $B$) extends uniquely to a continuous cross-section $\overline{a} : \bX \rightarrow \coprod_{y \in \bX} \Ab_y$ of the bundle associated with the continuous $C( \bX)$-algebra $(\Ab , \bX , \mu_{\Ab})$.
\end{enumerate}
Moreover, the following are equivalent:
\begin{enumerate}
\item[(iii)] $\Ab$ is canonically isomorphic to $C( \bX , B )$,
\item[(iv)] either $B$ is finite dimensional or $X$ is psuedocompact.
\end{enumerate}
\end{corollary}
\begin{proof}
(i): Let $a \in C^b (X, B)$ and $x \in X$.  If $U$ is a  neighbourhood of $x$ in $X$ with compact closure then there is $f \in C_0 (X)$ with $\restr{f}{\overline{U}} \equiv 1 $, so that $f \cdot a \in A$.  Moreover, for all $y \in U$ we have $\norm{(f\cdot a -a)(y)} = 0$, hence $a \in \gb (A)$.

Conversely, let $c \in \gb (A), x \in X$ and $\eps > 0$.  Then there is some $a \in A$ and a neighbourhood $U$ of $x$ such that
$
\norm{c(y)-a(y)} < \frac{\eps}{2}
$
for all $y \in U$.  Hence 
\[
\norm{c(y) - c(x)} \leq \norm{c(y)-a(y)}+ \norm{a(x)-c(x)} < \frac{\eps}{2} + \frac{\eps}{2} = \eps , 
\]
so that $y \mapsto c(y)$ is continuous at $x$.

(ii): By Theorem~\ref{t:ab}(iii), $(\Ab , \bX , \mu_{\Ab} )$ is a continuous $C( \bX)$-algebra.  The fact that every $a \in C^b(X,B)$ admits a continuous extension to $\overline{a} \in \Gamma^b ( (\Ab , \bX , \mu_{\Ab} ) )$ follows from Theorem~\ref{t:ext}. 

The equivalence of (iii) and (iv) follows from~\cite[Corollary 2]{williams_tensor}, which shows that the natural embedding $C( \bX , B ) \hookrightarrow C^b (X , B)$ is surjective if and only if either $B$ is finite dimensional or $X$ is pseudocompact.
\end{proof}

\begin{example}
Let $\A$ be the $C_0 ( \mathbb{N} )$-algebra $C_0 ( \mathbb{N} , B ) = c_0 (B)$.  Then $\Ab = C^b (\mathbb{N} , B ) = \ell^{\infty} (B)$, which defines a $C (\bX)$-algebra with respect to the natural multiplication of sequences by functions in $C(\bX)$.  For $y \in \bX \backslash X$, the fibre algebras $\Ab_y$ are given by ultrapowers of $B$.
\end{example}

\begin{remark}
Let $A=C_0(X,B)$ be a trivial $C_0(X)$-algebra.  Then Corollary~\ref{c:triv} shows that $\Ab = C^b(X,B)$  defines a trivial $C(\bX)$-algebra if and only if either $B$ is finite dimensional or $X$ is psuedocompact.  However, for general $B$, it is clear that $A$ admits a trivial compactification over $\bX$, namely $C( \bX , B)$.

Consider the case of a locally trivial $C_0(X)$-algebra $\A$ with constant fibre $C$.  If $C=M_n$ for some $n$, then $A$ is an $n$-homogeneous \Cst -algebra and $\Pm (A)$ is homeomorphic to $X$.  Moreover, in this case we have $A^b = M(A)$, and the following are shown to be equivalent in~\cite[Proposition 2.9]{phillips_recursive}:
\begin{enumerate}
\item[(i)] $A$ is of finite type, i.e., there exists a finite open cover $\{U_i : 1 \leq i \leq m \}$ of $X$ such that $\mu_A (C_0(U_i)) \cdot A \cong C_0 (U_i , C )$ for all $i$,
\item[(ii)] $\Abx$ is a locally trivial $C(\bX)$-algebra,
\item[(iii)] There exists a locally trivial compactification $\B$ of $\A$ over some compactification $K$ of $X$.
\end{enumerate}
It would be interesting to know whether or not the equivalence of (i) and (iii) still holds in the case of a locally trivial $C_0(X)$-algebra $\A$ with infinite dimensional fibre $C$. By considering the trivial case, Corollary~\ref{c:triv} shows that we cannot expect property (ii) to be equivalent to (i) and (iii) in general.  

It is clear, however, that (i) is a necessary condition for (iii).  Indeed, given such a $\B$, for each $y \in K$ let $V_y$ be an open neighbourhood of $y$ in $K$ such that $\mu_{B} \left( C_0 (V_y ) \right) \cdot B \cong C_0 (V_y , C )$.  Then since $K$ is compact, we can choose $y_1, \ldots , y_n \in K$ such that $V_{y_1}, \ldots , V_{y_n}$ cover $K$.  Setting $U_i = V_{y_i} \cap X$ for $1 \leq i \leq n$, the $\{ U_i : 1 \leq i \leq n \}$ cover $X$, and for each $i$, $C_0 ( X) \cdot C_0 ( V_{y_i} ) = C_0 (U_i )$. Moreover, together with Proposition~\ref{p:c0}(v), for each $i$ we have
\begin{align*}
\mu_A (C_0(U_i)) \cdot A & = \mu_A (C_0(U_i)) \cdot \left( \mu_B ( C_0 (X) ) \cdot B \right) \\
& = \mu_B \left( C_0 ( X) \right) \cdot \left( \mu_{B} \left( C_0 ( V_{y_i} ) \right) \cdot B \right) \\
& =  \mu_{B} \left( C_0(X) \right) \cdot C_0 ( V_{y_i} , C ) \\
& = C_0 (U_i , C ),
\end{align*}
so that $\A$ is of finite type.
\end{remark}

\section{The fibre algebras of $\Abx$}
\label{s:fibres}

In this section, we consider the question of characterising, for a compactification $\AbK$ of a $C_0(X)$-algebra $\A$, the set of points of $K$ for which the fibre algebras $\Ab_y$ of $\AbK$ are nonzero.  By analogy with the subset $X_A$ of $X$ defined in~\eqref{e:xa}, we denote this space by $K_{\Ab}$.  In other words, we consider the question of whether or not $K_{\Ab}$ is a compactification of $X_A$, and in particular, whether or not $\brac{ \beta X_A }_{\Ab}$ is the Stone-\v{C}ech compactification of $X_A$.

 We show in Theorem~\ref{t:nonzero} that this is the case for a large class of $C_0(X)$-algebras $\A$ (including all $\sigma$-unital continuous $C_0(X)$-algebras).  Moreover, we show in Theorem~\ref{t:notremote} that for all $\sigma$-unital $C_0(X)$-algebras, the set $K_{\Ab} \backslash X_A$ is at least dense in $K \backslash X_A$.

  In Example~\ref{e:remote1} however, we exhibit a separable $C_0 (X)$-algebra $\A$ and a point $y \in \beta X_A \backslash X_A$ for which $\Ab_y = \set{ 0 }$.  The key observation here is the fact that there exist remote points in the Stone-\v{C}ech remainder of $X_A$.

A key technique in establishing the above uses a deep technical result of Archbold and Somerset~\cite[Theorem 2.5]{arch_som_ideals}, applied in a similar manner to~\cite[Theorem 3.3]{arch_som_inner}.   We shall not need the full strength of this result here, Proposition~\ref{p:mult} below establishes a significant consequence needed for our purposes.

\begin{proposition}
\label{p:mult}
Let $\A$ be a $\sigma$-unital $C_0(X)$-algebra and $u$ a strictly positive element of $A$ with $\norm{u}=1$.  If $f \in C^b (X_A)$ with $0<f(x) \leq 1$ for all $x \in X_A$, then there is an element $b \in \Ab$ with $\norm{b(x)} \geq 1$ whenever $f(x) \leq \norm{u(x)}$.
\end{proposition}
\begin{proof}
Let $b \in M(A)$ be the element constructed from $u$ and $f$ in~\cite[Theorem 2.5]{arch_som_ideals}.  Then since $f$ is everywhere nonzero in $X_A$, we have $b \in A + H_x$ for all $x \in X_A$ by~\cite[Theorem 2.5 (ii)]{arch_som_ideals}, i.e., $b \in \Ab$.  

For each $x \in X_A$, let $g_{f(x)}:[0,1] \to [0,1]$ be the piecewise linear function with $g_{f(x)}(t)=0$ for $0 \leq t \leq \frac{1}{2} f(x)$, $g_{f(x)}(t)=1$ for $f(x) \leq t \leq 1$, and $g_{f(x)}$ linear on the interval $[\frac{1}{2} f(x), f(x)]$.  Then it is shown in~\cite[Theorem 2.5(i)]{arch_som_ideals} that
\[
\tilde{\pi}_x (b) = g_{f(x)} \left( \tilde{\pi}_x (u) \right)
\]
for all $x \in X_A$, where $\tilde{\pi}_x : M(A) \rightarrow M(A_x)$ is the strictly continuous extension of the quotient homomorphism $\pi_x : A \rightarrow A_x$ to $M(A)$ of~\eqref{e:pitilde}.  Since $u$ is positive, the norm of $\tilde{\pi}_x (u)$ is equal to  its spectral radius.  Hence if $f(x) \leq \norm{u(x)}$ ($=\norm{\tilde{\pi}_x (u)}$),  then $g_{f(x)} \brac{ \norm{ \tilde{\pi}_x (u) }} = 1$, and so the norm of the restriction of $g_{f(x)}$ to the spectrum of $\tilde{\pi}_x (u)$ is equal to $1$. In particular, $\norm{\tilde{\pi}_x (b)} = 1$.

Finally, using the fact that $\ker ( \tilde{\pi}_x ) \subseteq H_x$ for all $x \in X_A$, we see that 
\[\norm{b(x)} \geq \norm{\tilde{\pi}_x (b)} =1 \]
 whenever $f(x) \leq \norm{u(x)}$.

\end{proof}

Let $\A$ be a $\sigma$-unital $C_0(X)$-algebra and fix a strictly positive element $u \in A$ with $\norm{u}=1$.  Since norm functions $x \mapsto \norm{a(x)}$ vanish at infinity on $X$ for all $a \in A$, we may define for each real number $\alpha$ with $0 < \alpha \leq 1$ a compact subset $K(\alpha)$ of $X$ via
\begin{equation}
\label{e:kalpha}
K(\alpha) = \set{ x \in X : \norm{u(x)} \geq \alpha},
\end{equation}
(note that $K(\alpha) \subseteq X_A$).  Let $\set{ c_n }$ be a decreasing sequence, with $0 < c_n \leq 1$ for all $n$ and $c_n$ converging to $0$ as $n \to \infty$.  Then since $u$ is strictly positive, $u(x) \neq 0$ for all $x \in X_A$ and so
\[
\bigcup_{n=1}^\infty K(c_n) = X_A.
\]
For convenience we fix some notation for the remainder of this section; 
\begin{equation}
\label{e:kn}
K_n = K\left(\textstyle \frac{1}{n}\right)= \set{ x \in X : \norm{u(x)} \geq \frac{1}{n} },
\end{equation}
so that $\bigcup_{n=1}^{\infty} K_n = X_A$ as before.

The following Proposition is essentially an extension of the method of~\cite[Theorem 3.3]{arch_som_inner}.
\begin{proposition}
\label{p:reg}
Let $\A$ be a $\sigma$-unital $C_0(X)$-algebra and $u$ a strictly positive element of $A$ with $\norm{u}=1$. Let $K$ be a compactification of $X_A$, $\AbK$ the compactification of $\A$ defined in Theorem~\ref{t:ab}(ii), and $y \in K \backslash X_A$.

 Suppose that there is a (relatively) closed subset $F \subseteq X_A$ and a strictly decreasing sequence $\{ c_n \}$ with $0 < c_n \leq 1$ and $c_n \to 0$, such that
\begin{enumerate}
\item[(a)] $y \in \mathrm{cl}_{K} (F)$ and
\item[(b)] for all $n \in \mathbb{N}$, $F \cap K(c_n) \subseteq \mathrm{int}_{F} (F \cap K(c_{n+1}))$.
\end{enumerate}
Then there is $b \in \Ab$ such that the section $b:K \to \coprod_{y \in K} \Ab_y$ satisfies
\begin{enumerate}
\item[(i)] $\norm{b(x)}=1$ for all $x \in F$, and hence
\item[(ii)] $\norm{b(y)}=1$.
\end{enumerate}
In particular, $\Ab_y \neq \{ 0 \}$.
\end{proposition}
\begin{proof}
For each $n \in \mathbb{N}$ let $F_n  = F \cap K(c_n)$ (so that $F = \bigcup_{n=1}^{\infty} F_n$) and let $f_n : F \rightarrow [0,1]$ be a continuous function with $\restr{f_n}{F_n} \equiv 1$ and $\restr{f_n}{F \backslash \mathrm{Int}_F (F_{n+1})} \equiv 0$ (note that the assumption of (b) implies that the sets $F_n$ and $F\backslash \mathrm{Int} (F_{n+1})$ are disjoint and closed. Let $f = \sum_{n=1}^{\infty} 2^{-n} (c_n f_n)$, so that $f$ is continuous on $F$, $0 \leq f \leq 1$ and $f(x)>0$ for all $ x \in F$.

We claim that $f(x) \leq \norm{u(x)}$ for all $ x \in F$.  Indeed, any $x \in F$ belongs to $F_{k+1} \backslash F_k$ for some $k$, and hence $f_n (x) = 0 $ for $1 \leq n \leq k-1$.  It follows that
\begin{align*}
f(x) & = \sum_{n=k}^{\infty} 2^{-n} (c_n) f_n (x) \\
& \leq c_k\ \sum_{n=k}^{\infty} 2^{-n} f_n (x) \\
& \leq c_k 2^{-k+1} \leq c_k < \norm{u(x)},
\end{align*}
the final inequality holding since $x \not\in K(c_k)$ by assumption.

Now, since $X_A$ is normal and $F$ is closed, $f$ has a continuous extension to  $\overline{f} : X_A \rightarrow [0,1]$.  Suppose that $Z(\overline{f})$ is non-empty, then since $Z( \overline{f} )$ and $F$ are disjoint closed subsets of $X_A$, there is  a continuous function $k: X_A \rightarrow [0,1]$ such that $\restr{k}{ Z( \overline{f} ) } \equiv 1 $ and $\restr{k}{F} \equiv 0$.  If $Z( \overline{f} )= \emptyset$, then set $k=0$.

Finally, let $g = \min (\overline{f}+k , 1)$, so that $g:X_A \rightarrow [0,1]$, $g$ is continuous, $g(x)>0 $ for all $x \in X_A$ and $\restr{g}{F} = f$.

Using Proposition~\ref{p:mult} applied to $g$, we get $b \in \Ab$ with $\norm{b(x)}=1$ for all $x \in F$. Since $y \in \mathrm{cl}_{K} (F)$, this implies that for all neighbourhoods $W$ of $y$ in $K$ there is $x \in W \cap X_A$ with $\norm{b(x)}=1$. Moreover, as
\[
\norm{b(y)} = \inf_{W} \sup_{x \in W \cap X_A} \norm{b(x)},
\]
(where $W$ ranges over all neighbourhoods of $y$ in $K$) by Proposition~\ref{p:c0}(iii), it follows that $\norm{b(y)}=1$.  In particular, $\Ab_y \neq \{ 0 \}$.

\end{proof}

Proposition~\ref{p:reg} will be our main technique for constructing points $y \in K \backslash X_A$ for which $\Ab_y$ is nonzero.  In Theorem~\ref{t:nonzero} we shall show that in certain cases (such as that of continuous $C_0(X)$-algebras), we may take $F = X_A$, so that $\Ab_y \neq \{ 0 \}$ for all such $y$.  One of these cases arises when the base space $X$ is the so-called \emph{Glimm space} of $A$, a particular space constructed from $\Pm (A)$.

For a \Cst -algebra $A$, define an equivalence relation $\approx$ on $\Pm (A)$ as follows: for $P,Q \in \Pm (A)$, $P \approx Q$ if and only if $f(P) = f(Q)$ for all $f \in C^b ( \Pm (A) )$.  As a set, we define $\Gl (A)$ as the quotient space $\Pm (A) / \approx$, and we denote by $\rho_A : \Pm (A) \rightarrow \Gl (A)$ the quotient map.  For $f \in C^b ( \Pm (A)$, define $f^{\rho}$ on $\Gl (A)$ via $f^{\rho} ([P]) = f(p)$, where $[P]$ is the $\approx$ equivalence class of $P$ in $\Pm (A)$ (note that $f^{\rho}$ is well-defined by construction).  We then equip $\Gl (A)$ with the topology $\tau_{cr}$ induced by the functions $\set{ f^{\rho} : f \in C^b ( \Pm (A))}$.  With this topology, $\Gl (A)$ is a completely regular (Hausdorff) space (the \emph{complete regularisation} of $\Pm (A)$).  For more details of this construction, we refer the reader to~\cite[Chapter 3]{gill_jer},~\cite{arch_som_qs},~\cite{m_glimm} and~\cite[Chapter 2]{thesis}.

It is clear that if $\Gl (A)$ is locally compact, the continuous map $\rho_A : \Pm (A) \rightarrow \Gl (A)$ gives rise to a $C_0 ( \Gl (A) )$-algebra $\AGl$.  In general however, $\Gl (A)$ may fail to be locally compact, e.g.~\cite{dauns_hofmann}.  Nonetheless, if $\rho_A$ is regarded as a map $\Gl (A) \to \beta \Gl (A)$, we get a $C( \beta \Gl (A)$-algebra $\AbGl$.

The space $\Gl (A)$ is in some ways more tractable as a base space over which to represent a given \Cst-algebra $A$ as a $C_0(X)$-algebra.  Not every completely regular space $Y$ arises as $\Gl (A)$ for some \Cst -algebra $A$, indeed, Lazar and Somerset have recently given a complete characterisation (for separable $A$) of those spaces $Y$ that do~\cite{lazar_somerset}.  By contrast, every completely regular $\sigma$-compact space $Y$ arises as $X_A$ for some $C_0(X)$-algebra $\A$~\cite[Section 2]{arch_som_ideals}.

\begin{lemma}
\label{l:int}
Let $A$ be a $\sigma$-unital \Cst -algebra such that $\Gl (A)$ is locally compact.  Let $u \in A$ be a strictly positive element of norm 1. Then there is an increasing sequence $\{ n_j \}$ such that 
\begin{enumerate}
\item[(i)] $\displaystyle \bigcup_{j=1}^{\infty} K_{n_j} = \Gl (A)$, and
\item[(ii)] For each $j$ we have
\[
K_{n_j} \subset \mathrm{Int} K_{n_{j + 1} }.
\]
\end{enumerate}
\begin{proof}
Note that since $u$ is strictly positive it is evident that $\bigcup_{n=1}^{\infty} K_n = \Gl (A)$.  Moreover, since $\Gl (A)$ is locally compact, each $x \in \Gl (A)$ has an open neighbourhood $U_x$ in $\Gl (A)$ such that $\overline{U_x}$ is compact.  

By~\cite[Theorem 2.1]{lazar_glimm}, for each compact $K \subseteq \Gl (A)$ there is some $\alpha > 0$ such that
\[
K \subseteq \{ y \in \Gl (A) : \norm{u(y)} \geq \alpha \}.
\]
In particular, for each $x \in \Gl (A)$ there is $m_x \in \mathbb{N}$ such that $\overline{U}_x \subseteq K_{m_x}$.

We define the sequence $ \{ n_j \}$ inductively.  Let $n_1 = 1$. For $j \geq 1$, note that the collection $\{ U_x : x \in  K_{n_j+1} \}$ is an open cover of $ K_{n_j + 1}$, so by compactness there are $x_1, x_2 , \ldots , x_r \in K_{n_j+1}$ such that
\[
 K_{n_j+1} \subseteq \bigcup_{i=1}^r U_{x_i}.
\]
By the previous paragraph, there are $m_{x_i} \in \mathbb{N}$ with $\overline{U_{x_i}}  \subseteq K_{m_{x_i}}$ for $1 \leq i \leq r$.  Set $n_{j+1} = \max \{ n_{x_i} : 1 \leq i \leq r \}$.  Then
\[
 K_{n_j+1} \subseteq \bigcup_{i=1}^r U_{x_i} \subseteq \bigcup_{i=1}^r \overline{U}_{x_i} \subseteq K_{n_{j+1}}.
\]
Properties (i) and (ii) are then immediate.
\end{proof}
\end{lemma}
When $X_A$ is not $\Gl (A)$, the conclusion of Lemma~\ref{l:int} can fail, even when $X_A$ is compact and $A$ commutative, as Example~\ref{e:nokn} shows.  The reason that this situation does not arise in the case where $X_A = \Gl (A)$ is the result of Lazar~\cite{lazar_quot}, which shows that the usual topology on $\Gl(A)$ is precisely the quotient topology induced by the canonical surjection $\Pm (A) \rightarrow \Gl(A)$ when $A$ is $\sigma$-unital.

\begin{example}
\label{e:nokn}
Let $A=c_0 = C_0 ( \mathbb{N})$, $X = \set{0} \cup \set{\frac{1}{n} : \in \Nat }$ with the subspace topology from $\mathbb{R}$. Let $\phi_A : \Nat \rightarrow X$ the continuous surjection defined by $\phi(1)=0$ and $\phi(n)=\frac{1}{n-1}$ for $n \geq 2$.  Then we get a $C(X)$-algebra $\A$ with base map $\phi_A$ such that $X_A = X$. To avoid ambiguity we will write elements $a \in A$ as sequences $\set{ a_n }$.

Let $u \in A$ be the strictly positive element $u_n = \frac{1}{n}$.  Then $u(\frac{1}{n}) = u_{n+1} = \frac{1}{n+1}$ for $n \in \Nat$, and $u(0) = u_1=1$.  It follows that $K_1 = \set{0}$ and 
\[
K_n = \set{0} \cup \set{\frac{1}{m} : 1 \leq m \leq n-1 } 
\]
for $n \geq 2$.  Note that the point $0$ is not an interior point of any set $K_n$.
\end{example}

\begin{theorem}
\label{t:nonzero}
Let $\A$ be a $C_0(X)$-algebra, and suppose that one of the following conditions hold:
\begin{enumerate}
\item[(i)] there is a strictly positive element $u \in A$ with $x \mapsto \norm{u(x)}$ continuous on $X_A$ (e.g. if $\A$ is a $\sigma$-unital, continuous $C_0(X)$-algebra),
\item[(ii)] for all $x \in X_A$ we have
\[
\mu_A \left( C_0(X) \right) \cap A \not\subseteq J_x,
\]
\item[(iii)] $\A = \AGl$, where $\Gl (A)$ is locally compact and $A$ $\sigma$-unital.
\end{enumerate}
Then there is $b \in \Ab$ with $\norm{b(x)}=1$ for all $x \in X_A$. Hence if $K$ is any compactification of $X_A$, the fibre algebras $\Ab_y$ of the compactification $\AbK$ of $\A$ are nonzero  for all $y \in K$. 
\end{theorem}
\begin{proof}
Suppose that $A$ satisfies (i), then we may assume w.l.o.g. that $\norm{u}=1$.  For each $n \in \mathbb{N}$, let $K_n$ be the subset of $X_A$ defined by the norm function of $u$ as in~\eqref{e:kn}.  Since $x \mapsto \norm{u(x)}$ is continuous on $X$, for all $n \in \mathbb{N}$ the set 
\[
O_{n+1}:=\set{ x \in X_A: \norm{u(x)} > \frac{1}{n+1} }
\]
is open, and  by definition we have $K_n \subseteq O_{n+1} \subseteq K_{n+1}$.  Moreover, since $u$ is strictly positive, $X_A = \bigcup_{n=1}^{\infty} K_n$.  Hence we may apply Proposition~\ref{p:reg} with $F=X_A$ and $c_n = \frac{1}{n}$ for all $n$.

If (ii) holds, then by Proposition~\ref{p:abiff}(i), we have $\Ab = M(A)$.  In particular, this implies that $\Pm (\Ab)$ is compact, and hence its continuous image $K_{\Ab}$ is again compact. Since $X_A$ is dense in $K$, it follows that $K_{\Ab} = K$, i.e.  $\Ab_y \neq \{ 0 \}$ for all $y \in K$, the conclusion follows.  Note that we may take $b=1_{M(A)} \in \Ab$ in this case.

In case (iii), again let $u $ be a strictly positive element of $A$ with $\norm{u}=1$.  Let $\{n_j \}_{j \geq 1}$ be the sequence constructed in Lemma~\ref{l:int}, so that again, $K_{n_j} \subseteq \mathrm{Int}_{X_A} (K_{n_{j+1}})$ for all $j \geq 1$ and $X_A = \bigcup_{j=1}^{\infty} K_{n_j}$.  Deleting repetitions where necessary, we get a strictly decreasing sequence $c_n$, $0<c_n \leq 1$, converging to $0$ such that $K(c_n) \subseteq \mathrm{Int} K(c_{n+1})$ for all $n$ and $X_A = \bigcup_{n=1}^{\infty} K(c_n)$.  Applying Proposition~\ref{p:reg} with $F=X_A$ yields the required result.

\end{proof}

Example~\ref{e:notremote} gives a $C(X)$-algebra $\A$ (with $X$ compact and $X_A$ dense in $X$) which fails to satisfy any of the conditions (i), (ii) or (iii) of Theorem~\ref{t:nonzero}, yet for which $\Ab_y \neq \{ 0 \}$ for all $y \in X \backslash X_A$.

We remark also that Example~\ref{e:notcts} (which is a small modification of Example~\ref{e:notremote}), shows that condition (i) of Theorem~\ref{t:nonzero} is strictly weaker than continuity.

\begin{example}
\label{e:notremote}
Let $A=C_0 ( \Nat , K(H) )$ and identify $\Pm (A)$ with $\Nat$ in the usual way. Let $X=[0,1]$ and let $\phi_A : \mathbb{N} \rightarrow \mathbb{Q} \cap (0,1)$ be a bijection.  Then $\A$ defines a $C(X)$-algebra with base map $\phi_A$.  We shall show that the fibre algebras $\Ab_y$ of the compactification $(\Ab,X,\mu_{\Ab}^X)$ of $\A$ are nonzero for all $y \in X$.

In this case, $\Gl (A) = \Pm (A)$ and so $X_A = \mathbb{Q} \cap (0,1)$ is not homeomorphic to $\Gl(A)$, so that condition (iii) of Theorem~\ref{t:nonzero} does not apply.  Since $Z(A) = \{ 0 \}$, we have 
\[ \mu_A \left( C_0 (X) \right)  \cap A \subseteq ZM(A) \cap A = Z(A) = \{ 0 \} \subseteq J_x \]
for all $x \in X_A$, so that (ii) does not apply either.

If $a \in A$ is any nonzero element, then $x \mapsto \norm{a(x)}$ is discontinuous on $X_A$.  Indeed, suppose that $\norm{a(y)} \neq 0 $ for some $y \in X_A$ and that $x \mapsto \norm{a(x)}$ were continuous at $y$.  Then the set $\{ x \in X_A : \norm{a(x)} > \frac{1}{2} \norm{a(y)} \}$ would be open in $X_A$ and contained in the compact subset $\{ x \in X_A : \norm{a(x)} \geq \frac{1}{2} \norm{a(y)} \}$.  The latter would then be a compact subset of $\mathbb{Q}$ with non-empty interior, which is a contradiction.  In particular, condition (i) cannot apply.

Nonetheless, let $y \in X \backslash X_A $, and let $q_n$ be a sequence of distinct rationals in $(0,1)$ converging to $y$. Then $F:=\{ q_n: n \in \mathbb{N} \}$ is (relatively) closed and discrete in $X_A$.  Since $K_n$ is compact for each $n$, each $K_n \cap F$ is finite, and so we may apply Proposition~\ref{p:reg} to $F$ to conclude that $\Ab_y \neq \{ 0 \}$.

\end{example}

\begin{example}
\label{e:notcts}
Let $\A$ be as in Example~\ref{e:notremote} and let $B$ be the \Cst-subalgebra $A+\mu_A\left( C_0 (0,1) \right)$ of $M(A)$.  Note that $M(B)=M(A)$ and that $(B,X,\mu_B)$ is a $C(X)$-algebra with $\mu_B = \mu_A$. Let $\set{f_n} \in C_0 \left( (0,1) \right)$ be an increasing approximate identity for $C_0 (0,1)$ with $\norm{f_n}=1$ for all $n$.  Then $\mu_A(f_n)$ is an approximate identity for $B$ and so $f:=2^{-n} f_n$ is a strictly positive element of $B$. Clearly $y \mapsto \norm{f(y)}$ is continuous on $X = X_B$, while as in Example~\ref{e:notremote}, no element $a$ of the subalgebra $A$ of $B$ has $y \mapsto \norm{a(y)}$ continuous on $X$.
\end{example}

In fact, the techniques used in Example~\ref{e:notremote} may be used to obtain a much stronger result, as Theorem~\ref{t:notremote} shows.

\begin{theorem}
\label{t:notremote}
Let $\A$ be a $\sigma$-unital $C_0(X)$-algebra, $K$ a compactification of $X_A$ and $\AbK$ the compactification of $\A$ given in Theorem~\ref{t:ab}(ii).  Denoting by $\set{ \Ab_y : y \in K }$ the fibres of $\AbK$, then the set
\[
\left\{ y \in K \backslash X_A : \Ab_y \neq \{ 0 \} \right\}
\]
is dense in $K \backslash X_A$.  
\end{theorem}
\begin{proof}
Let $ w \in K \backslash X_A$ and let $W$ be a compact neighbourhood of $w$ in $K$.  We shall show that there is a point $y \in W \cap \brac{ K \backslash X_A }$ such that $\Ab_y \neq 0$. As $w \in  \mathrm{cl}_{K}(X_A)$, $W \cap X_A \neq \emptyset$ and has the following properties:
\begin{enumerate}
\item[(i)] $W \cap X_A$ is closed, and hence $\sigma$-compact,
\item[(ii)] Since $w \in \mathrm{Int}_{K} (W) \cap \mathrm{cl}_{K} (X_A)$, it follows that $w \in \mathrm{cl}_{K} (W \cap X_A)$. Hence $W \cap X_A$ is not closed in $K$ and in particular,  is non-compact.
\item[(iii)] Since a $\sigma$-compact, completely regular space is pseudocompact if and only if it is compact~\cite[8.2 and 8A]{gill_jer}, we conclude from (i) and (ii) that $W \cap X_A$ is not pseudocompact. 
\end{enumerate}

By~\cite[Corollary 1.21]{gill_jer}, $W \cap X_A$ contains a countably infinite, discrete subset $F = \{ x_n : n \geq 1 \}$ which is C-embedded in $W \cap X_A$.  Suppose $x \in W \cap X_A$ lies in the closure of $F$ in $W \cap X_A$, and let $f:F \rightarrow \mathbb{R}$ be given by $f(x_n)=n$. Since $F$ is C-embedded in $W \cap X_A$, $f$ has a continuous extension to $\overline{f} \in C(W \cap X_A)$.  Thus if $x$ were not in $F$, every neighbourhood of $x$ in $W \cap X_A$ would contain a subset on which $f$ were unbounded, contradicting the fact that $\overline{f}$ is continuous, real-valued and extends $f$.  Hence $x \in F$ and so $F$ is closed in $W \cap X_A$, and moreover, closed in $X_A$.

Since $F \subseteq W \cap X_A$ and $F$ is non-compact, we have
\[
\mathrm{cl}_{K} (F) = \mathrm{cl}_W (F) \supsetneq F,
\]
and hence $\mathrm{cl}_{W}(F) \backslash F$ is nonempty and contained in $W \backslash X_A$.  Let $y \in \mathrm{cl}_{W}(F) \backslash F$.

  Note that for each $n \in \mathbb{N}$, $F \cap K_n$ must be finite (since $F$ is closed and discrete and $K_n$ is compact).  In particular,  for all $n$ we have  $F \cap K_n \subseteq \mathrm{int}_{F} (F \cap K_{n+1})$.  By Proposition~\ref{p:reg}, $ \Ab_y \neq \{ 0 \}$.

\end{proof}

The proof of Theorem~\ref{t:notremote} shows that $\Ab_y \neq \{ 0 \}$ whenever $y \in K \backslash X_A$ lies in the closure in $K$ of a relatively closed, discrete subset of $X_A$ (countability is in fact ensured by $\sigma$-compactness of $X_A$).  In general, however, there do exist points of $K \backslash X_A$  that do not lie in the closure of any subset of this form.  In the context of Stone-\v{C}ech remainders such points are called \emph{far points}, and are a particular case of \emph{remote points}.  

There exist remote points of $\beta \mathbb{R}$ and $\beta \mathbb{Q}$~\cite{van_douwen}. In the case of $\beta \mathbb{R}$, this was originally shown (assuming the continuum hypothesis) by Fine and Gillman in~\cite{fine_gilmann}, and later (without this assumption) by van Douwen~\cite{van_douwen}. 

Example~\ref{e:remote1}  exhibits a $C_0(X)$-algebra $\A$ with $\Ab_y= \set{ 0 }$ for every remote point $y \in \beta X_A \backslash X_A$.

\begin{example}
\label{e:remote1}
Let $A=C_0 ( \Nat , K(H) )$ and identify $\Pm (A)$ with $\Nat$ in the usual way. Let $X=\beta \mathbb{Q}$ and let $\phi_A : \mathbb{N} \rightarrow \mathbb{Q}$ be a bijection.  Then $\A$ defines a $C(X)$-algebra with base map $\phi_A$ and $X_A = \mathbb{Q}$.

We claim that for any $b \in \Ab$ and $\eps > 0$, the set
\[
\set{q \in \mathbb{Q} : \norm{b(x)} \geq \eps }
\]
is discrete.  Writing $a=(a_n)$ for an element of $A$ (i.e. a sequence of elements of $K(H)$), the set 
\[
\set{ n \in \mathbb{N} : \norm{a_n} \geq \eps }
\]
is finite, and hence its image 
\[ K(a,\eps):= \set{ q \in \mathbb{Q} : \norm{a(q)} \geq \eps } \subset \mathbb{Q} \]
is again finite.

Now, if $b \in \gb (A)$ with and $q \in \mathbb{Q}$, there is $a \in A$ and a neighbourhood $U$ of $q$ in $\mathbb{Q}$ such that $\norm{b(x)-a(x)} < \frac{1}{2} \eps$ for all $x \in U$.  Moreover, since $K(a,\eps)$ is finite, there is a neighbourhood $V \subseteq U$ of $q$ with $\norm{a(x)} < \frac{1}{2} \eps $ for all $x \in V \backslash \{ q \}$.

In particular, for all $x \in V \backslash \set{q}$, we have
\begin{align*}
\norm{b(x)} = \norm{b(x)-a(x)+a(x)} & \leq \norm{b(x)-a(x)}+\norm{a(x)} \\
& < \frac{\eps}{2} + \frac{\eps}{2} = \eps.
\end{align*}

Since by Theorem~\ref{t:ext}, $\Ab$ is isomorphic to $\gb (A)$, it follows that the set 
\[
\set{ q \in \mathbb{Q} : \norm{b(q)} \geq \eps }
\]
is discrete for all $\eps > 0$.

Now suppose that $y \in \beta \mathbb{Q} \backslash \mathbb{Q}$ is a remote point, and $b \in \Ab$.  By Proposition~\ref{p:c0}(iii), we have
\[
\norm{b(y)} = \inf_W \sup_{x \in W \cap \mathbb{Q} } \norm{b(x)},
\]
as $W$ ranges over all neighbourhoods of $y$ in $\beta \mathbb{Q}$.  Since $y$ is remote, for all $\eps > 0$ there is a neighbourhood $W$ of $y$ such that $W$ is disjoint from the countable, discrete set $\set{q \in \mathbb{Q} : \norm{b(x)} \geq \eps }$.  Hence $\sup_{x \in W \cap \mathbb{Q}} \norm{b(x)} < \eps$, and so $\norm{b(y)}=0$ for all $b \in \Ab$.

Note that any set closed subset $F \subseteq X_A$ satisfying the hypothesis of Proposition~\ref{p:reg} (for any choice of decreasing sequence $\{c_n \}$) must be discrete ( $F$ is necessarily countable since $X_A$ is). 

Indeed, for any such $F$ we have $F_n : = F \cap K(c_n)$ finite.  Since $F_n \subseteq \mathrm{Int}_F (F_{n+1})$, it is clear that each element of $F_n$ has a neighbourhood disjoint from $F \backslash F_{n+1}$.  Moreover, if $x \in F_{n}\backslash F_{n-1}$, then since $\bigcup_{m=1}^{n} F_m$ is finite, $x$ has a neighbourhood disjoint from $\brac{\bigcup_{m=1}^n F_m} \backslash \set{x}$ also.

\end{example}

\begin{remark}
If $A$ and $\phi_A$ are as in Example~\ref{e:remote1}, then regarding $\phi_A$ as a map into $\mathbb{R}$, $(A , \mathbb{R} , \mu_A)$ is a $C_0(\mathbb{R})$-algebra, and $\beta \mathbb{R}$ is a compactification of $X_A$.  For each $b \in \Ab$ and $\eps >0$, we have
\[
\set{ x \in \mathbb{R} : \norm{b(x)} \geq \eps } = \set{ q \in \mathbb{Q} : \norm{b(q)} \geq \eps }
\]
and hence this set is discrete.  Since there exist remote points of $\beta \mathbb{R}$, the $C( \beta \mathbb{R})$-algebra $(\Ab , \beta \mathbb{R} , \mu_{\Ab} )$ has fibres $\Ab_y = \set{ 0 }$ at these points.

\end{remark}

\section{Maximality and uniqueness of $\Abx$}
\label{s:univ}
For a completely regular space $X$, $\bX$ may be described as the unique compactification of $X$ satisfying the following maximality condition: given any compactification $K$ of $X$, the inclusion $\iota :X \rightarrow K$ has a unique extension to a continuous surjection $\overline{\iota} : \bX \rightarrow K$~\cite[Theorem 6.12]{gill_jer}.  Equivalently, given any such $K$, there is a unital, injective $\ast$-homomorphism $C(K) \to C( \bX)$. In this section we study the problem of generalising  the latter property to the $C( \bXA)$-algebra $\Abx$ associated with a $C_0(X)$-algebra $\A$.

Indeed, suppose that $(\Psi, \psi ) : \A \rightarrow \B$ is a compactification of the $C_0(X)$-algebra $\A$.   Theorem~\ref{t:u1} below shows that the $C( K )$-algebra $\AbK$ satisfies a similar maximality condition, namely, that there is an injective morphism $\B \to \AbK$ which is the identity on $\A$. 

If in addition $\B$ is a continuous $C(K)$-algebra (note that a necessary condition for this to occur is that $\A$ be a continuous $C_0(X)$-algebra), then we get an injective morphism $\B \to \Abx$.  This can fail without continuity of $\B$, as we shall see in Example~\ref{e:ninj}.

\begin{theorem}
\label{t:u1}
Let $(B,K,\mu_B)$ be a compactification of $(A,X,\mu_A)$.  
\begin{enumerate}
\item[(i)] There is an injective $\ast$-homomorphism $B \rightarrow \Ab$ extending the identity on $A$.
\item[(ii)] Let $\AbK$ denote the compactification of $\A$ given in Theorem~\ref{t:ab}(ii).  Then $B \to \Ab$ is $C(K)$-linear and hence gives rise to an injective morphism $ \B \to \AbK$.
\item[(iii)] The pair of $\ast$-homomorphisms $( B \to \Ab , C(K) \to C( \beta X_A ))$ define a morphism $\B \rightarrow (\Ab , \beta X_A, \mu_{\Ab} )$.  If in addition $X_A=X$ and $\B$ is a continuous $C(K)$-algebra, then $(B \to \Ab, C(K) \to C ( \bX ) )$ is an injective morphism.
\end{enumerate}
\end{theorem}
\begin{proof}
(i) Since $A$ is an essential ideal of $B$ we may regard $B$ as a \Cst -subalgebra of $M(A)$ containing $A$.  Hence it suffices to prove that $B \subseteq \Ab$.  But then since $\B$ is a compactification of $\A$, we may identify $B$ with a \Cst -subalgebra of cross-sections in $\gb (A)$ by Theorem~\ref{t:ext}(i).  But then $\Ab$ is isomorphic to $\gb (A)$ under the same map by Theorem~\ref{t:ext}(iii), so that $B \subseteq \Ab$ as required.

(ii) Let $a \in A$ and factorise $a$ as $a = \mu_A^K (f) a_0$ for some $f \in C(K)$.  Then since $\AbK$ and $\B$ are both compactifications of $\A$ we have
\begin{align*}
\mu_B(g)a = \mu_B (g) \brac{ \mu_A^K (f) a_0 } & = \mu_A^K (gf) (a_0) \\
& = \mu_{\Ab}^K(g) \brac{ \mu_A (f)(a_0) } \\
& = \mu_{\Ab}^K (g) a.
\end{align*}
It follows that $\mu_B(g)b=\mu_{\Ab}^K (g)b$ for all $g \in C(K)$ and $b \in B$, hence $(B \to \Ab, \mathrm{Id}_{C(K)})$ is a  morphism.
  
  Finally, $( B \hookrightarrow \Ab , \mathrm{id}_{C(K)} ) : \B \rightarrow \AbK$ is injective by Proposition~\ref{p:inj}. 

(iii) Since by part (ii), $(B \to \Ab , \mathrm{id}_{C (K)})$ is a morphism it is clear that $(B \to \Ab , C(K) \to C( \beta X_A))$ is again a morphism.

Suppose now that $X_A=X$ and that $\B$ is a continuous $C(K)$-algebra.  For clarity, we shall denote the given $\ast$-homomorphisms by $\Psi : B \to \Ab$ and $\psi : C(K) \to C( \bX)$.  Then the dual of $\psi$ is $\psi^* : \bX \rightarrow K$, the canonical continuous map extending the identity on $X$.

Since $\Abx$ and $\B$ are both compactifications of $A$, the induced $\ast$-homomorphisms $\Psi_x: B_x \rightarrow \Ab_x$ are in fact $\ast$-isomorphisms  for all $x \in X$.  Thus it remains to show that the $\ast$-homomorphisms $\Psi_y : B_{\psi^*(y)} \rightarrow \Ab_y$ are injective for all $y \in \bX \backslash X$.

  Indeed, given such a $y$, let $(x_{\alpha})$ be a net in $X$ converging to $y$ in $\bX$, so that, regarding $X$ as a subspace of $K$, the same net $(x_{\alpha})$ converges to $\psi^* (y)$.  Suppose for a contradiction that $\Psi_y$ were not injective. Then there would be some $b \in B$ with $\norm{b(\psi^*(y))}=1$ for which $\Psi_y \brac{b(\psi^{\ast}(y))} = \brac{\Psi(b)}(y)=0$.  Since $\B$ is a continuous $C(K)$-algebra, $t \mapsto \norm{b(t)}$ is lower-semicontinuous on $K$, so that for any $\eps >0$ there is an index $\alpha_0$ with $\norm{b(x_{\alpha})} > 1- \eps$ whenever $\alpha \geq \alpha_0$.
  
  On the other hand, since $y \mapsto \norm{\Psi (b)(y)}$ is upper-semicontinuous on $\bX$, there is some index $\alpha_1$ with $\norm{\Psi(b) (x_{\alpha})}< \eps$ for all $ \alpha \geq \alpha_1$.  But then $\norm{b(x)} = \norm{\Psi (b)(x)}$ for all $x \in X$. Choosing $\alpha \geq \max \{ \alpha_0,\alpha_1 \}$ yields a contradiction.

\end{proof}

The assumption that $\B$ is a continuous $C(K)$-algebra in Theorem~\ref{t:u1}(ii) cannot be dropped in general, as is easily seen from the commutative case:
\begin{example}
\label{e:ninj}
Let $\A$ be the $C_0 ( \Nat )$-algebra defined by $C_0 ( \Nat )$, and let $\nhat = \Nat \cup \{ \infty \}$  be the one point compactification of $\Nat$.  Let $B = C^b ( \Nat )$, and equip $B$ with the structure of a $C(\nhat)$-algebra with respect to the obvious action of $C( \nhat )$ on $C^b ( \Nat )$ by multiplication (equivalently, define the base map $\Pm(C^b ( \Nat )) \cong \beta \Nat \rightarrow \nhat$ to be the canonical surjection).  Then $(B,\nhat, \mu_B)$ is a discontinuous $C( \nhat )$-algebra, and is clearly a compactification of $\A$.

In this case $(\Ab, \bX, \mu_{\Ab} )$  is canonically isomorphic to $C(\beta \Nat )$.  The fibre algebras of $B$ are $B_n \cong \Cp$ for $n \in \Nat$ and $B_{\infty} \cong (\ell^{\infty}/c_0)$, while the fibre algebras of $(\Ab, \bX, \mu_{\Ab} )$ are isomorphic to $\Cp$ for all $y \in \beta \Nat$.  Hence $B_{\infty}$ does not embed into $\Ab_y$ for any $y \in \beta \Nat \backslash \Nat$.  In this case the homomorphisms $\Psi_y : B_{\infty} \rightarrow \Ab_y, y \in \beta \Nat \backslash \Nat$, coincide with the point-evaluations.

If $C = C( \nhat )$ is regarded as a bundle compactification of $\A$, then $C$ is a continuous $C(\nhat)$-algebra and is identified with the subalgebra of $C( \beta \Nat )$ consisting of those functions that are constant on $\beta \Nat \backslash \Nat$.  In this case the fibre maps send $C_{\infty} \cong \Cp \rightarrow \Ab_y \cong \Cp$ canonically for $y \in \beta \Nat \backslash \Nat$.

\end{example}

\begin{lemma}
\label{l:sigma}
Let $\A$ be a $\sigma$-unital, continuous $C_0(X)$-algebra and let $f \in C^b (X_A)$ be real-valued and non-negative. Then there is some $c \in \Ab$ with $\norm{c(x)}=f(x)$ for all $x \in X_A$.
\end{lemma}
\begin{proof}
Let $b$ be the element of $\Ab$ with $\norm{b(x)}=1$ for all $x \in X_A$ obtained from Theorem~\ref{t:notremote}. Then $f$ extends to $\overline{f} \in C ( \bXA)$.  Setting $c = \mu_{\Ab}(\overline{f}) b$, we have $c \in \Ab$ and
\[
\norm{c (x)} = \norm{f(x) b(x)} = \abs{f(x)} \norm{b(x)} = f(x)
\]
for all $x \in X_A$.
\end{proof}

Theorem~\ref{t:unique} considers the question of whether $\Abx$ is unique amongst compactifications $\B$ of $A$ having the property that $\Gamma (B) = \gb (\Ab)$.  Note that by Theorem~\ref{t:ext}, we cannot expect this in general (as $X_A$ often admits more than one compactification).  

If in addition $\A$ is continuous, then it turns out that $\Abx$ is the unique continuous compactification of $\A$ with the required extension property.

\begin{theorem}
\label{t:unique}
Let $\A$ be a $C_0(X)$-algebra and $\B$ a compactification of $\A$. Suppose that $\B$ has the property that for every $a \in \gb(A)$ there is $b \in \Gamma (B)$ such that $\restr{b}{X}=a$. Then
\begin{enumerate}
\item[(i)] the injective morphism $\B \rightarrow \AbK$ of Theorem~\ref{t:u1} is an isomorphism.
\end{enumerate}
If in addition, $\A$ is continuous and $\sigma$-unital, then
\begin{enumerate}
\item[(ii)] $\B$ is continuous if and only if $K$ is canonically homeomorphic to $\beta X_A$.  Moreover, when this occurs, the injective morphism $\B \to \Abx$ is an isomorphism.
\end{enumerate}
\end{theorem}
\begin{proof}
(i): We first show that the injective $\ast$-homomorphism $B \to \Ab$ is surjective. If any $a \in \gb (A)$ extends uniquely to $b \in \Gamma(B)$, it then follows that the injective $\ast$-homomorphism $B \rightarrow \gb (A)$  sending $ b \mapsto \restr{b}{X_A}$ of Theorem~\ref{t:ext}  is a $\ast$-isomorphism.  In the commutative diagram
\[
\xymatrix{
B \ar[r] \ar[d] & \Ab \ar[d] \\
\Gamma (B) \ar_{b \mapsto \restr{b}{X_A}}[r] & \gb (A), \\
}
\]
the vertical arrows are isomorphisms by Theorems~\ref{t:go} and~\ref{t:ext}(ii) respectively. Since $b \mapsto \restr{b}{X}$ is a $\ast$-isomorphism, it must be the case that   $B \rightarrow \Ab$ is also an isomorphism.

Since the homomorphism $C(K) \to C(K)$ is the identity, it follows that $\mu_B = \mu_{\Ab}^K$ (identifying $B$ with $\Ab$ and hence $ZM(B)$ with $ZM(\Ab)$).  In particular, $\B$ may be identified with  $\AbK$.

(ii): Suppose first that $\B$ is continuous.

If $f \in C^b (X_A)_+$ then by Lemma~\ref{l:sigma}(iii), there is $c \in \gb (A)$ with $f(x) = \norm{c(x)}$ on $X_A$.  By assumption, the extension property of $\B$ would then imply that there is an element $b \in B$ with $b(x)  = c(x)$ for all $x \in X_A$.  Then since $\B$ is a continuous $C(K)$-algebra, $y \mapsto \norm{ b(y) }$ is a continuous function on $K$ extending $f$.  Setting $\overline{f}(y)=\norm{b(y)}$ for $y \in K$, it follows that every $f \in C^b (X_A)_+$ has an extension to a continuous function $\overline{f}$ in $C(K)$.

For general $g \in C^b (X_A)$, write $g=g_+-g_-$ where $g_+$ and $g_-$ are the positive and negative parts of $g$ respectively. Then $\overline{g}=\overline{g_+}-\overline{g_-}$ gives the required extension.  Thus $X_A$ is \Cst -embedded in $K$ and hence $K$ is canonically homeomorphic to $\beta X_A$.

Conversely, suppose that  $K $ is canonically homeomorphic to $\beta X_A$.  By Theorem~\ref{t:ab}(iii), $\B$ is continuous.

The final assertion then follows from Proposition~\ref{p:inj}.
\end{proof}

Note that one choice of compactification of $X_A$ is given by taking $K= \mathrm{cl}_{\bX} (X_A)$.

\begin{corollary}
\label{c:cts}
Let $\A$ be a $\sigma$-unital, continuous $C_0(X)$-algebra.  If $K = \mathrm{cl}_{\bX}(X_A)$ then $\AbK$ is a continuous $C(K)$-algebra if and only if $X_A$ is C$^{\ast}$-embedded in $X$.
\end{corollary}
\begin{proof}
If $X_A$ is C$^{\ast}$-embedded in $X$ then $K$ is canonically homeomorphic to $\beta X_A$.  It then follows that $\AbK$ is canonically isomorphic to $\Abx$ and hence is continuous by Theorem~\ref{t:ab}(iii).  Conversely, suppose that $\AbK$ is a continuous $C( K)$-algebra. Then by Theorem~\ref{t:unique}(ii), $K$ is canonically homeomorphic to $\bXA$, and so $X_A$ is  \Cst -embedded in $\beta X$ and moreover, \Cst -embedded in $X$.
\end{proof}

Since for any compact space $K$ it is clear that $\beta K = K$, the operation of constructing $\bX$ from $X$ is a `closure' operation in the sense that $\beta ( \bX) = \bX$.  It is natural to ask whether or not constructing $\AbK$ from $\A$ also has this property.  Indeed, applying Theorem~\ref{t:ab} to $\AbK$  gives rise to a compactification $(\Abb, K , \mu_{\Abb}^K )$ of $\AbK$,   where $A$ is $C(K)$-linearly embedded into $\Abb$ as an essential ideal, hence $(\Abb, K , \mu_{\Abb}^K )$ is also a compactification of $\A$.

Certainly we have $\Ab = A$ whenever $X_A$ is compact by Proposition~\ref{p:abiff}.  However, this alone is not sufficient to show that $\Abb = \Ab$, since as we have seen in Section~\ref{s:fibres}, the set of nonzero fibres of $\AbK$ need not be compact.  Nonetheless, the maximality of $\AbK$ obtained in Theorem~\ref{t:unique} ensures that we do indeed have $\Abb = \Ab$ , without any additional assumptions on $\A$, as Corollary~\ref{c:closure} shows.

\begin{corollary}
\label{c:closure}
Let $\A$ be a $C_0(X)$-algebra and consider the compactification $\AbK$ of $\A$, where $K$ is any compactification of $X_A$.  Then the $C(K)$-linear inclusion of $\Ab$ into $\Abb$ is an isomorphism $\AbK \to (\Abb, K , \mu_{\Abb}^K )$.
\end{corollary}
\begin{proof}
By construction, the embedding of $\A$ into $(\Abb, K , \mu_{\Abb}^K )$  is a compactification of $\A$, and we have natural isomorphisms $\gb (A) \equiv \gb (\Ab ) \equiv \gb (\Ab)$ by Theorem~\ref{t:ext}.  By Theorem~\ref{t:unique}, it must be the case that $\AbK \to (\Abb, K , \mu_{\Abb}^K )$, i.e. the inclusion of $\Ab$ into $\Abb$ is surjective.
\end{proof}

\begin{remark}
Note that Proposition~\ref{p:abiff} can not be used here to conclude that $\brac{\bXA}_{\Ab}$ is a compactification of $X_A$, since $\Ab$ will not in general be $\sigma$-unital.  Indeed, let $\A$ be the $C(X)$-algebra of Example~\ref{e:remote1}.  Then since $X_A \subseteq X_{\Ab} \subsetneq X$ and $X_A$ is dense in $X$, it follows that $X_{\Ab}$ cannot be compact.  Nonetheless, Corollary~\ref{c:closure} shows that $\Abb = \Ab$.  In particular, the equivalence of~\cite[Theorem 3.3]{arch_som_inner} may fail when the $C_0(X)$-algebra under consideration is not $\sigma$-unital.

\end{remark}

\end{document}